\documentclass[11pt]{article}
\usepackage{colordvi,epsfig,amssymb}
\usepackage{amsmath}
\usepackage[table]{xcolor}   
\usepackage[algosection,linesnumbered, ruled,noend]{algorithm2e}
\usepackage{verbatim,tikz}
\usepackage{mathtools}  
\usepackage{caption,subcaption}   
\graphicspath{{./figures/}}   
\usepackage{fullpage}
\makeatletter
\newcommand{\leqnomode}{\tagsleft@true}
\newcommand{\reqnomode}{\tagsleft@false}
\makeatother

\newlength\myheight
\newlength\mydepth
\settototalheight\myheight{Xygp}
\usepackage{graphicx,calc}
\newcommand{\ignore}[1]{}
\newcommand{\R}{\mathbb{R}}
\newcommand{\N}{\mathbb{N}}

\newcommand{\qed}{\nobreak \ifvmode \relax \else
      \ifdim\lastskip<1.5em \hskip-\lastskip
      \hskip1.5em plus0em minus0.5em \fi \nobreak
      \vrule height0.75em width0.5em depth0.25em\fi}
\providecommand{\colvec}[1]{{\text{\scriptsize $
\begin{array}{c}
 #1
\end{array}$
}}}
\newcommand{\eqdef}{\overset{\text{def}}{=}} 

\newcommand{\E}[1]{\mathbf{E}\left[#1\right] } 
\newcommand{\norm}[1]{\lVert#1\rVert}
\newcommand{\dotprod}[1]{\left< #1\right>}
\newcommand{\Tr}[1]{\mathbf{Tr}\left( #1\right)}
\providecommand{\Null}[1]{\mathbf{Null}\left( #1\right)}
\providecommand{\Rank}[1]{\mathbf{Rank}\left( #1\right)}
\providecommand{\Range}[1]{\mathbf{Range}\left( #1\right)}

\newtheorem{definition}{Definition}
\newtheorem{theorem}[definition]{Theorem}

\newtheorem{lemma}[definition]{Lemma}

\usepackage[colorinlistoftodos,bordercolor=orange,backgroundcolor=orange!20,linecolor=orange,textsize=scriptsize]{todonotes}

\usepackage{hyperref}

\title{Linearly Convergent Randomized Iterative Methods \\ for Computing the Pseudoinverse} 
\author{Robert M. Gower\thanks{\it Inria - ENS, \tt robert.gower@inria.fr} \and Peter Richt\'{a}rik\thanks{{\it School of Mathematics, University of Edinburgh}. This author would like to acknowledge support from the EPSRC Grant EP/K02325X/1, {\em Accelerated Coordinate Descent Methods for Big Data Optimization} and the EPSRC Fellowship EP/N005538/1, {\em Randomized Algorithms for Extreme Convex Optimization.} }\\\\
  }

\begin{document}
\maketitle


\begin{abstract}
We develop the first stochastic incremental method for calculating the Moore-Penrose pseudoinverse of a real matrix. By leveraging three alternative characterizations of pseudoinverse matrices, we design three  methods for calculating the pseudoinverse: two general purpose methods and one  specialized to symmetric matrices. The two general purpose methods are proven to converge linearly to the pseudoinverse of any given matrix. For calculating the pseudoinverse of full rank matrices we present two additional specialized methods which enjoy a faster convergence rate than the general purpose methods. We also indicate how to develop randomized methods for calculating approximate range space projections, a much needed tool in inexact Newton type methods or quadratic solvers when linear constraints are present. Finally, we present numerical experiments of our general purpose methods for calculating pseudoinverses and show that our methods greatly outperform the Newton-Schulz method on large dimensional matrices. 
\end{abstract}
{\footnotesize {\bf MSC classes:} 	15A09, 15B52, 15A24, 65F10, 65F08, 68W20, 65Y20, 65F20, 68Q25, 68W40, 90C20\\
{\bf ACM class:} 	G.1.3}
\section{Introduction}

Calculating the pseudoinverse matrix is a basic numerical linear algebra tool required throughout scientific computing; for example in neural networks~\cite{Tapson2013}, signal processing~\cite{Robertson_apseudoinverse,Feichtinger1991} and image denoising~\cite{ADLER_ICASSP_2010}. Perhaps the most important application of approximate pseudoinverse matrices is in preconditioning; for instance, within the approximate inverse preconditioning\footnote{A more accurate name for these techniques would be ``approximate \emph{pseudo}inverse preconditioning''. This is because they form a preconditioner by approximately solving $\min_X\norm{AX-I}$, with $A$ not always guaranteed to be nonsingular, leading to the solution being  the pseudoinverse.} techniques~\cite{Grote96,Gould1998,Chow1998,Benzi1999}.

 Currently, the pseudoinverse matrix is calculated using either the singular value decomposition or, when the dimensions of the matrix are large, using a Newton-Schulz type method~\cite{Ben-Israel1965,Ben-Israel1966b,PanSchreiber91}. However, neither of these aforementioned methods were designed with big data problems in mind, and when applied to matrices from machine learning, signal processing and image analysis, these classical methods can fail due exceeding the cache memory or take too much time. 
 
 In this paper we develop new fast  stochastic incremental methods for calculating the pseudoinverse, capable of calculating an approximate pseudoinverse of truly large dimensional matrices. The problem of determining the pseudoinverse from stochastic measurements also serves as a model problem for determining an approximation to a high dimensional object from a few low dimensional measurements.

The new stochastic methods we present are part of a growing class of ``sketch-and-project''  methods~\cite{GowerThesis}, which have successfully been used for solving linear systems~\cite{Gower2015,Gower2015c}, the distributed average consensus problem~\cite{Loizourichtarik2016,Gower2015c} and inverting matrices~\cite{Gower2016}. 

We  also envision that the new methods presented here for calculating pseudoinverse matrices will lead to the development of new quasi-Newton methods,  much like the development of new randomized methods for inverting matrices~\cite{Gower2016} has lead to the development of new stochastic quasi-Newton methods~\cite{GowerGold2016}.

\subsection{The Moore-Penrose Pseudoinverse}

The pseudoinverse of a real rectangular matrix $A\in \R^{m\times n}$ was first defined  as the unique matrix $A^\dagger\in \R^{n\times m}$  that satisfies four particular matrix equations~\cite{Penrose1955,Moore1920}. 
However, for our purposes it will be more convenient to recall a definition using the singular value decomposition (\emph{SVD}). Let $A = U \Sigma V^\top$ be the SVD of $A$, where $U\in \R^{m \times m}$ and $V\in \R^{n \times n}$ are orthogonal matrices and $\Sigma \in \R^{m\times n}$ is a diagonal matrix. The pseudoinverse $A^\dagger \in \R^{n \times m}$ is defined as $A^\dagger = V \Sigma^\dagger U^\top$, where $\Sigma^\dagger$ is the diagonal matrix with $\Sigma^\dagger_{ii} = 1/\Sigma_{ii}$ if $\Sigma_{ii} \neq 0$ and $\Sigma^\dagger_{ii} =0$ otherwise. This immediately gives rise to a method for calculating the pseudoinverse via the SVD decomposition which costs $O(m^2 n)$ floating point operations. When $m$ and $n$ are both large, this can be exacerbating and also unnecessary if one only needs a rough approximation of the pseudoinverse.  Therefore, in this work we take a different approach. 

In particular, it turns out that the pseudoinverse can alternatively be defined as the least-Frobenius-norm solution to any one of the three equations given in Lemma~\ref{lem:pseudo3}.
\begin{lemma} \label{lem:pseudo3} The pseudoinverse matrix $A^{\dagger}$ is the least-Frobenius-norm solution of any of the three equations:
\[(P1)\quad AX A = A, \qquad (P2)\quad A^\top  = XAA^\top, \qquad \text{and} \qquad  (P3)\quad A^\top  = A^\top AX. \] 
\end{lemma}
In Sections~\ref{sec:ATAX} and~\ref{sec:AXA}
we prove this statement for $(P1)$ and $(P2)$, and $(P3),$ respectively.

  We use  these  three variational characterizations of the pseudoinverse given in Lemma~\ref{lem:pseudo3}  to design three different stochastic iterative methods for calculating the pseudoinverse. Based on  (P2) and (P3), we propose two  methods for calculating the pseudoinverse of any real matrix in Section~\ref{sec:ATAX}. We exploit the symmetry in (P1) to propose a new randomized method for calculating the pseudoinverse of a symmetric matrix in Section~\ref{sec:AXA}.

In the next lemma we collect several basic properties of the pseudoinverse which we shall use often throughout the paper.
\leqnomode 
\begin{lemma} Any matrix  $A\in \R^{m\times n}$ and its pseudoinverse $A^\dagger \in \R^{n\times m}$ satisfy the following identities:\label{lem:pseudo} 
\begin{align} 
(A^{\dagger})^\top &= (A^\top)^\dagger \label{it:consprop0}  \hspace{10cm}\\
A^\top  &= A^{\dagger}AA^\top  \label{it:consprop2} \\
A^\top  &= A^\top AA^{\dagger}  \label{it:consprop3} \\
\Null{A^\dagger} & = \Null{A^\top}  \label{it:consprop4}
\end{align}
\end{lemma}
\reqnomode
Note that in the  identities above the pseudoinverse \emph{behaves} like the inverse would, were it to exist.
Because of~\eqref{it:consprop0}, we will use $A^{\dagger \top}$  to denote $(A^{\dagger})^\top $ or $(A^\top)^\dagger$.  Lemma~\ref{lem:pseudo} is a direct consequence of the definition of the pseudoinverse;  see~\cite{Penrose1955,Moore1920} for a proof based on the classical definition and~\cite{Desoer1963} for a proof based on a definition of the pseudoinverse through projections (all of which are equivalent approaches).

\subsection{Notation} We denote  the  Frobenius inner product and norm by
 \[\dotprod{X, Y}  \eqdef  \Tr{X^\top Y} \qquad \mbox{and} \qquad \norm{X} = \sqrt{\Tr{X^\top X}},\] 
where $X$ and $Y$ are any compatible real matrices and $\Tr{X}$ denotes the trace of $X$. Since the trace is invariant under  cyclic permutations, for matrices $X,Y,Z$ and $W$ of appropriate dimension, we have 
\begin{equation} \label{eq:traceperm}
\dotprod{X,YZW} = \Tr{ X^\top YZW} = \Tr{(WX^\top Y) Z} = \dotprod{Y^\top X W^\top ,Z}. \end{equation}
By $\Null{A}$ and $\Range{A}$ we denote the null space and range space of $A$, respectively.
For a positive semidefinite matrix $G$, let $\lambda_{\min}^+(G)$ denote the smallest nonzero eigenvalue of $G$.

\section{Sketch-and-Project Methods Based on~(P3) and (P2) } \label{sec:ATAX}

In view of property~(P3) of Lemma~\ref{lem:pseudo3},  the pseudoinverse can be characterized as the solution to the  constrained optimization problem 
\begin{equation}\label{eq:ATAX}
A^{\dagger} \eqdef \arg \min \frac{1}{2}\norm{X}^2,\quad  \mbox{subject to} \quad A^\top  =  A^\top AX.
\end{equation}
We shall prove in Theorem~\ref{theo:ATAX} that the above variational characterization has the following equivalent dual formulation
\begin{equation} \label{eq:ATAXdual}
A^{\dagger} = \arg_X \min_{X,\Gamma} \frac{1}{2}\norm{X-A^\dagger }^2,\quad  \mbox{subject to} \quad X  = A^\top A \Gamma.
\end{equation}

The dual formulation~\eqref{eq:ATAXdual} appears to be rather impractical, since using~\eqref{eq:ATAXdual} to calculate  $A^\dagger$ requires projecting the unknown matrix $A^\dagger$ onto a particular affine matrix space. But duality reveals that~\eqref{eq:ATAXdual} can be calculated by solving the primal formulation~\eqref{eq:ATAX}, which does not require knowing the pseudoinverse a priori. The dual formulation reveals that we should not search for $A^\dagger$ within the whole space $\R^{n \times m}$ but rather, the pseudoinverse is contained in the matrix space which forms the constraint in~\eqref{eq:ATAXdual}.

In the next section we build upon the characterization~\eqref{eq:ATAX} to develop a new stochastic method for calculating the pseudoinverse. 
\subsection{The method}

Starting from an iterate $X_k \in \R^{n \times m}$, we calculate the next iterate $X_{k+1}\in \R^{n \times m}$ by drawing a random matrix $S \in \R^{m \times \tau}$ from a fixed distribution $\mathcal{D}$ (we do not pose any restrictions on $\tau$) and projecting $X_k$ onto the sketch of~(P3):
\begin{equation}\label{eq:SATAX}
X_{k+1} \eqdef \arg \min \frac{1}{2}\norm{X-X_k}^2,\quad  \mbox{subject to} \quad S^\top A^\top  = S^\top A^\top AX.
\end{equation}
The dual formulation of~\eqref{eq:SATAX} is given by
\begin{equation} \label{eq:SATAXdual}
X_{k+1} \eqdef \arg_X \min_{X,\Gamma} \frac{1}{2}\norm{X-A^\dagger }^2,\quad  \mbox{subject to} \quad X  = X_k+A^\top A S \Gamma.
\end{equation}
The duality of these two formulations is established in the following theorem, along with an explicit solution to~\eqref{eq:SATAX} that will be used to devise efficient implementations.
\begin{theorem} \label{theo:ATAX}
If $X_k \in \Range{A^\top A}$  then solving~\eqref{eq:ATAX} and~\eqref{eq:SATAX} is equivalent to solving~\eqref{eq:ATAXdual} and~\eqref{eq:SATAXdual}, respectively.
Furthermore,  the explicit solution to~\eqref{eq:SATAX} is given by
\begin{equation} \label{eq:SATAXsol}
\boxed{X_{k+1} =X_k-A^\top A S(S^\top A^\top A A^\top AS)^{\dagger}S^\top A^\top(AX_k-I)}.
\end{equation}
\end{theorem}
\begin{proof} We will first show, using Lagrangian duality,  that~\eqref{eq:SATAX} and~\eqref{eq:SATAXdual} are equivalent.
The Lagrangian of~\eqref{eq:SATAX} is given by
\begin{eqnarray}
 L(X,\Gamma) &=& \frac{1}{2}\norm{X-X_k}^2 + \dotprod{\Gamma, S^\top A^\top AX-S^\top A^\top } \nonumber \\
 & = &  \frac{1}{2}\norm{X-X_k}^2 + \dotprod{A^\top AS\Gamma, X }
 - \dotprod{\Gamma, S^\top A^\top}. \label{eq:LagrangianSATAX}
\end{eqnarray}
Since~\eqref{eq:SATAX} is a convex optimization problem, strong duality implies that 
\[ \eqref{eq:ATAXdual} \quad = \quad \max_\Gamma \min_X  L(X,\Gamma)  = \min_X \max_\Gamma L(X,\Gamma) =  \eqref{eq:ATAX}. \]
Differentiating~\eqref{eq:LagrangianSATAX} in $X$ and setting to zero gives
\begin{equation}\label{eq:XATASGam}
X = X_k-A^\top A S \Gamma.
\end{equation}
Since $X_k \in \Range{A^\top A}$, using the above we conclude  that $X_{k+1}\in \Range{A^\top A} $.
Left multiplying by $S^\top A^\top A$ and observing the constraint in~\eqref{eq:SATAX} gives
$S^\top A^\top  
= S^\top A^\top AX_k -S^\top A^\top A A^\top AS \Gamma.$
Thus,
\begin{equation}\Range{\Gamma} \subset (S^\top A^\top A A^\top AS)^{\dagger}S^\top A^\top(AX_k-I) + \Null{S^\top A^\top A A^\top AS}. \label{eq:adna9sdnh981}
\end{equation}
From  Lemma~\ref{lem:09709s} with $G = AA^\top$ and $W= AS$ we have that
$ \Null{S^\top A^\top A A^\top AS} = \Null{A^\top AS}.$
Consequently, left multiplying~\eqref{eq:adna9sdnh981} by $A^\top A S$ gives
\begin{eqnarray}
A^\top A S\Gamma & = & A^\top A S(S^\top A^\top A A^\top AS)^{\dagger}S^\top A^\top(AX_k-I). 
\end{eqnarray}
The above combined with~\eqref{eq:XATASGam} gives~\eqref{eq:SATAXsol}.

To derive the dual of~\eqref{eq:SATAX}, first substitute~\eqref{eq:XATASGam} into~\eqref{eq:LagrangianSATAX} 
\begin{eqnarray}
 L(X,\Gamma) &= &  \frac{1}{2}\norm{X-X_k}^2 - \dotprod{X-X_k, X }
 - \dotprod{\Gamma, S^\top A^\top A A^\dagger} \nonumber\\
 & =&-\frac{1}{2}\norm{X-X_k}^2  - \dotprod{X-X_k, X_k }
  - \dotprod{A^\top AS\Gamma, A^\dagger} \nonumber \\
  & = & -\frac{1}{2}\norm{X-X_k}^2 
 + \dotprod{X-X_k,  A^\dagger-X_k}  \pm \frac{1}{2}\norm{A^\dagger-X_k}^2 \nonumber \\
 & =& -\frac{1}{2}\norm{X-X_k - (A^\dagger -X_k) }^2 +\frac{1}{2}\norm{A^\dagger-X_k}^2.
\end{eqnarray}
Calculating the argument that maximizes the above, subject to the constraint~\eqref{eq:XATASGam}, is equivalent to solving~\eqref{eq:SATAXdual}. Thus~\eqref{eq:SATAX} and~\eqref{eq:SATAXdual} are dual to one another and consequently equivalent.

Finally,  to see that~\eqref{eq:ATAXdual} is the dual of~\eqref{eq:ATAX}, note that by substituting $X_k =0$ and $S =I$ into~\eqref{eq:SATAX} and~\eqref{eq:SATAXdual} gives
~\eqref{eq:ATAX} and~\eqref{eq:ATAXdual}, respectively. Furthermore, when $S =I$ in ~\eqref{eq:SATAXsol} we have that
\begin{eqnarray*}
 X_{k+1} &\overset{\eqref{eq:SATAXsol}+\text{Lemma}~\ref{lem:pseudo} }{=}&X_k-A^\top A  (A^\top A A^\top A)^{\dagger}A^\top A(X_k-A^\dagger)\\
 & \overset{P1}{=}&  X_k -(X_k-A^\dagger) = A^\dagger,
\end{eqnarray*}
where in the last equality we used that $X_k -A^{\dagger} \in \Range{A^\top A}$.
Consequently the pseudoinverse is indeed the solution to~\eqref{eq:ATAX} and~\eqref{eq:ATAXdual}.
\hfill \qed
%
%
%
\end{proof}
   
The bottleneck  in computing~\eqref{eq:SATAXsol} is performing the matrix-matrix product $S^\top A$, which costs $O(\tau m n)$ arithmetic operations. Since we allow $\tau$ to be any positive integer, even $\tau =1$, the iterations~\eqref{eq:SATAXsol} can be very cheap to compute. Furthermore, method~\eqref{eq:SATAX} converges linearly (in $L2$) under very weak assumptions on the distribution $\cal D$, as we show in the next section.

\subsection{Convergence}
\label{subsec:ATAXconv}
Since the iterates~\eqref{eq:SATAXsol} are defined by a projection process, as we shall see,  proving convergence is rather straightforward.
Indeed, we will now prove that the iterates~\eqref{eq:SATAXsol} converge in $L2$ to the pseudoinverse; that is, the expected norm of $X_{k}-A^{\dagger}$ converges to zero. 
Furthermore, we have a precise expression for the rate at which the iterates converge.

The proofs of convergence of all our methods follow the same machinery.
We first start by proving an invariance property of the iterates; namely, that all the iterates reside in a particular affine matrix subspace. We then show that $X_k - A^{\dagger}$ converges to zero within the said matrix subspace. 

\begin{lemma}\label{lem:invariant} 
If  $
\;\Range{X_0} \subset \Range{A^\top A},$ then the iterates~\eqref{eq:SATAXsol} are such that \\$\Range{X_k -A^{\dagger}}\subset\Range{A^\top A}$ for all $k.$
\end{lemma}
\begin{proof}
Using induction and the constraint in~\eqref{eq:SATAXdual} we have that $\Range{X_k} \subset \Range{A^\top A}$ for all $k.$ The result now follows from $\Range{A^\dagger} \subset \Range{A^\top A}$ as can be seen in~\eqref{eq:ATAXdual}. \hfill \qed
 \end{proof}

\begin{theorem}
Let $X_0 \in \R^{n\times m}$ with  $\Range{X_0}\subset \Range{A^\top A}$ and let $H_S \eqdef  S(S^\top A^\top A A^\top AS)^{\dagger}S^\top.$  The expected iterates~\eqref{eq:SATAXsol} evolve according to
\begin{equation} \label{eq:SATAXExp}
\E{X_{k+1}-A^{\dagger}} = \E{I-A^\top A H_S A^\top A}\E{X_k-A^{\dagger}}.
\end{equation}
Furthermore, if $\E{H_S}$ is finite and positive definite then 
\begin{equation}\label{eq:SATAXExpnorm}
\E{\norm{X_{k}-A^{\dagger}}^2} \leq \rho^k \,\norm{X_0-A^{\dagger}}^2,
\end{equation}
where 
\begin{equation} \label{eq:rhoSATAX}
\rho = 1- \lambda_{\min}^+\left(A^\top A \E{H_S} A^\top A\right).
\end{equation}
\end{theorem}

\begin{proof}
Let $R_k \eqdef X_{k}-A^{\dagger}$ and 
$Z \eqdef A^\top A S (S^\top (A^\top A)^2S)^{\dagger}S^\top A^\top A = A^\top A H_S A^\top A.$ Subtracting $A^{\dagger}$ from both sides of~\eqref{eq:SATAXsol} we have
\begin{equation} \label{eq:SATAfixedp}
 R_{k+1}= (I-Z)R_k.
 \end{equation}
Taking expectation conditioned on $R_k$ gives
$\E{R_{k+1} \, | \, R_k}= (I-\E{Z})R_k.$
Taking expectation again gives~\eqref{eq:SATAXExp}.   Using the properties of pseudoinverse, it is easy to show that $Z$ is a projection matrix and thus $I-Z$ is also a projection matrix\footnote{See Lemma~2.2 in~\cite{Gower2015} for an analogous proof}.
Taking norm squared and then expectation conditioned on $R_k$ in~\eqref{eq:SATAfixedp} gives
\begin{eqnarray}
\E{\norm{R_{k+1}}^2 \, | \, R_k } &=& \E{\dotprod{(I-Z)R_k, (I-Z)R_k}\, | \, R_k } \nonumber \\
& \overset{(I-Z) \text{ is a proj.}}{=}& \E{\dotprod{(I-Z)R_k,R_k}\, | \, R_k } \nonumber\\
& = & \norm{R_k}^2 - \dotprod{\E{Z}R_k,R_k}. \label{eq:SATAtheo1}
\end{eqnarray}
From Lemma~\ref{lem:invariant} we have that there exists $W_k$ such that $R_k = A^\top A W_k.$  Therefore,
\begin{eqnarray} 
\dotprod{\E{Z}R_k,R_k} & \overset{\text{Lemma~\ref{lem:invariant}}}{=} & \dotprod{A^{\top}A\E{Z}A^\top A W_k, W_k}  \nonumber \\
& =&  \dotprod{(A^{\top}A)( A^\top A) H_S (A^\top A)(A^\top A) W_k, W_k} \nonumber\\
& \overset{\text{Lemma~\ref{lem:WGWtight} } }{\geq} & 
\lambda_{\min}^+\left(  A^\top A H_S A^\top A\right) \dotprod{(A^\top A) W_k,(A^{\top}A) W_k} \nonumber \\
& =& \lambda_{\min}^+\left(  A^\top A H_S A^\top A\right)  \norm{R_k}^2 = (1-\rho) \, \norm{R_k}^2. \label{eq:SATAtheo2}
\end{eqnarray} 
Taking expectation in~\eqref{eq:SATAtheo1} we have
\[\E{\norm{R_{k+1}}^2} = \E{\norm{R_k}^2} - \E{\dotprod{\E{Z}R_k,R_k} } \overset{\eqref{eq:SATAtheo2} }{\leq} \rho \,\E{\norm{R_k}^2}.  \]
It remains now to unroll the above recurrence to arrive at~\eqref{eq:SATAXExpnorm}. \hfill \qed
\end{proof}
   
With a precise expression for the convergence rate~\eqref{eq:rhoSATAX} opens up the possibility of tuning the distribution of $S$ so that the resulting has a faster convergence. Next we give an instantiation of the method~\eqref{eq:SATAXsol} and indicate how one can choose the distribution of $S$ to accelerate the method.  We refer to methods based on~\eqref{eq:SATAXsol} as the \texttt{SATAX} methods, inspired on the constraint in~\eqref{eq:SATAX} whose right hand side almost spells out \texttt{SATAX}. Later in Section~\ref{sec:numerics} we perform experiments on  variants of the \texttt{SATAX} method.
\subsection{Discrete examples}
\label{subsec:SATAXexe}

Though our framework and Theorem~\ref{theo:ATAX} allows for $\mathcal{D}$ to be a continuous distribution, for illustration purposes here we focus our attention on  developing examples where $\mathcal{D}$ is a discrete distribution. 

For any discrete distribution $\mathcal{D}$ the random matrix $S \sim \mathcal{D}$ will have a finite number of possible outcomes. Fix $r$ as the number of outcomes and  let $\tau \in \N$ and  $S= S_i \in \R^{m \times \tau}$ with probability  $p_i>0$ for $i=1,\ldots, r$. Let  $\mathbb{S} \eqdef [S_1, \ldots, S_r] \in \R^{m\times r\tau}$. 
If 
\[p_i = \frac{\Tr{S_i^\top (A^\top A)^2S_i}}{\Tr{\mathbb{S}^\top  (A^\top A)^2 \mathbb{S}}},\]
then by Lemma~\ref{lem:convprob} with $G = A^\top A$ as proven in the Appendix, the rate of convergence in Theorem~\ref{theo:ATAX} is given by
\begin{equation} \label{eq:rhoATAXexe}
\rho = 1- \lambda_{\min}^+(\E{Z}) \leq 1- \frac{\lambda_{\min}^+ (\mathbb{S}^\top  (A^\top A)^2 \mathbb{S})}{\Tr{\mathbb{S}^\top  (A^\top A)^2 \mathbb{S}}}. 
\end{equation}
The number $\left.\lambda_{\min}^+ (\mathbb{S}^\top  (A^\top A)^2 \mathbb{S})\right/\Tr{\mathbb{S}^\top  (A^\top A)^2 \mathbb{S}}$ is known as the scaled condition number of $\mathbb{S}^\top  (A^\top A)^2 \mathbb{S}$ and it is the same condition number on which the rate of convergence of the randomized Kaczmarz method depends~\cite{Strohmer2009}. 
This rate~\eqref{eq:rhoATAXexe} suggests that we should choose $\mathbb{S}$ so that $\mathbb{S}^\top (A^\top A)^2 \mathbb{S}$ has a concentrated spectrum and consequently, the scaled condition number is minimized. Ideally, we would want 
$\mathbb{S} = (A^\top A)^{\dagger},$ but this is not possible in practice, though it does inspire the following heuristic choice. If we choose $r$ so that $r\tau =n$, then we can set
 $\mathbb{S} = X_kX_k^\top,$ and consequently $X_k \rightarrow A^\dagger$ we have that  $\mathbb{S} \rightarrow A^\dagger A^{\dagger \top} = (A^\top A)^{\dagger}.$ Though through experiments  we have identified that choosing the sketch matrix so that $\mathbb{S} = X_k$ resulted in the best performance. This observation, together with other empirical observations, has lead us to suggest two alternative sketching strategies: 
\paragraph{Uniform $\tau$--batch sampling:} We say $S$ is a uniform $\tau$--batch sampling if $\mathbf{P}(S = I_{:C}) = \left.1 \right/ \binom{n}{|C|}$ where $C \in \{1,\ldots, n\}$ is a random subset with $|C| = \tau$ chosen uniformly at random and $I_{:C}$ denotes the column concatenation of the columns of the identity matrix indexed by $C.$
\paragraph{Adaptive sketching:} Fix the iteration count $k$  and consider the current iterate $X_{k}$. We say that $S$ is an adaptive sketching if $\mathbf{P}(S = X_kI_{:C})$ where  $I_{:C}$ is a uniform $\tau$--batch sampling.

When using a uniform $\tau$--batch sampling together with the \texttt{SATAX} method, we refer to the resulting method as the \texttt{SATAX\_uni}. We use \texttt{SATAX\_ada} when referring to the method that uses the adaptive sketching. We benchmark both these methods later in Section~\ref{sec:numerics}.

\subsection{A sketch-and-project method based on~(P2) } \label{sec:XAAT}
Yet another characterization of the pseudoinverse  is given by the solution to the  constrained optimization problem based on property (P2):
\begin{equation}\label{eq:XAAT}
A^{\dagger} \eqdef \arg \min \frac{1}{2}\norm{X}^2,\quad  \mbox{subject to} \quad A^\top  =  XAA^\top.
\end{equation}
which has the following equivalent dual formulation
\begin{equation} \label{eq:XAATdual}
A^{\dagger} \eqdef \arg_X \min_{X,\Gamma} \frac{1}{2}\norm{X-A^\dagger }^2,\quad  \mbox{subject to} \quad
X  = \Gamma^\top A A^\top.
\end{equation}
%
Transposing the constraint in~\eqref{eq:XAAT} gives 
$A = AA^\top X^\top .$
Since the Frobenius norm is invariant to transposing the argument, we have that by setting $Y = X^\top$ in~\eqref{eq:XAAT} we get
\begin{equation}\label{eq:XAATT}
A^{\dagger\top} \eqdef \arg \min \frac{1}{2}\norm{Y}^2,\quad  \mbox{subject to} \quad A = AA^\top Y.
\end{equation}
It is now clear to see that~\eqref{eq:XAATT} is equivalent to~\eqref{eq:ATAX} where each occurrence of $A$ has been swapped for $A^\top.$ Because of this simple mapping from~\eqref{eq:XAATT} to~\eqref{eq:ATAX} we refrain from developing methods based on~\eqref{eq:XAATT} (even though these methods are different). 

\section{A Sketch-and-Project Method Based on~(P1) } \label{sec:AXA}
Now we turn out attention to designing a method based on~(P1). In contrast  with the development in the previous section, here we make explicit use of the symmetry present in~(P1). In particular, we introduce a novel sketching technique which we call  \emph{symmetric sketch}.  As we shall see, if $A$ is symmetric, our  method \eqref{eq:SAXASsol} maintains the symmetry of iterates if started from a symmetric matrix $X_0$.  Throughout this section we assume that $A \in \R^{n\times n}$ is a symmetric matrix.

The final variational characterization of the pseudoinverse from Lemma~\ref{lem:pseudo3}, based on (P1), is
\begin{equation}\label{eq:AXA}
A^{\dagger} = \arg \min \frac{1}{2}\norm{X}^2,\quad  \mbox{subject to} \quad AXA = A.
\end{equation}
As before, we have the following equivalent dual formulation
  \begin{equation} \label{eq:AXAdual}
A^{\dagger} = \arg_X   \min_{\Gamma, X} \frac{1}{2} \norm{ A^{\dagger}-X }^2 \quad \mbox{subject to} \quad X = A\Gamma A.
  \end{equation}
 
 In Section~\ref{sec:mehod} we describe our method.  In Theorem~\ref{theo:SAXAS} we prove that these two formulations are equivalent and also show that the iterates of our method are symmetric. This is in contrast with techniques such as the block BFGS update and other methods designed for calculating the inverse of a matrix in~\cite{Gower2016}, where symmetry has to be imposed on the iterates through an explicit constraint.  
  
  Calculating approximations of the pseudoinverse of a symmetric matrix is particularly relevant when designing variable metric methods in optimization, where one wishes to maintain an approximate of the (pseudo)inverse of the Hessian matrix. In contrast to the symmetric methods for calculating the inverse presented in~\cite{Gower2016}, which can be readily interpreted as extensions of known quasi-Newton methods, the method presented in this section appears not to be related to any Broyden quasi-Newton method~\cite{Broyden1967}, nor the SR1 update. This naturally leads to  the question: how would a quasi-Newton method based on~\eqref{eq:SAXASsol} fair? We leave this question  to future research.

\subsection{The method}\label{sec:mehod}

Similarly to the methods developed in Section~\ref{sec:ATAX}, we define an iterative method by projecting onto a sketch of~\eqref{eq:AXA}. In this case, however, we use the {\em symmetric sketch}.  Specifically, we calculate the next iterate $X_{k+1}$ via
\begin{equation}\label{eq:SAXAS}
X_{k+1} \eqdef \arg \min \frac{1}{2}\norm{X-X_k}^2,\quad  \mbox{subject to} \quad S^\top AXAS = S^\top AS,
\end{equation}
where  $S \in \R^{m \times \tau}$ is drawn from  $\mathcal{D}$. The dual formulation is given by
  \begin{equation}\label{eq:SAXASdual}
 X_{k+1} = \arg_X\min_{\Gamma, X} \frac{1}{2} \norm{ X - A^{\dagger} }^2 \quad \mbox{subject to} \quad X = X_k+ A S\Gamma S^\top A.
  \end{equation}
This symmetric sketch makes its debut in this work, since it has not been used in any of the previous works developing sketch-and-project sketching methods~\cite{Gower2015,Gower2015c,Gower2016}. 
  
\begin{theorem} \label{theo:SAXAS}
Solving~\eqref{eq:AXA} and~\eqref{eq:SAXAS} is equivalent to solving~\eqref{eq:AXAdual} and~\eqref{eq:SAXASdual}, respectively.
Furthermore, the explicit solution to~\eqref{eq:SAXAS} is
\begin{equation}\label{eq:SAXASsol}
\boxed{X_{k+1} = X_k +   A S (S^\top A^2 S)^{\dagger} S^\top ( A - AX_kA) S (S^\top A^2 S)^\dagger S^\top A.}
\end{equation}
\end{theorem}
\begin{proof}
Let 
\begin{equation} \label{eq:EandB}
E \eqdef X-X_k \quad \mbox{and}  \quad B \eqdef S^\top(A-AX_kA)S.
\end{equation}
Using the above renaming we have that~\eqref{eq:SAXAS} is equivalent to solving
\begin{equation}\label{eq:SAXASE}
 \min \frac{1}{2}\norm{E}^2,\quad  \mbox{subject to} \quad S^\top AEAS = B.
\end{equation}
The Lagrangian of~\eqref{eq:SAXASE} is given by
\begin{eqnarray}
 L(E,\Gamma) &=& \frac{1}{2}\norm{E}^2 + \dotprod{\Gamma, S^\top AEAS-B} \nonumber\\
 & \overset{\eqref{eq:traceperm} }{=} & \frac{1}{2}\norm{E}^2 + \dotprod{A S\Gamma S^\top A,  E} - \dotprod{\Gamma,B}.\label{eq:Lagrangianeq:SAXASE}
\end{eqnarray}
Differentiating in $E$ and setting the derivative to zero gives
\begin{equation}\label{eq:LagSAXASEsoltemp} 
E = -A S\Gamma S^\top A.\end{equation}
Left and right multiplying by $S^\top A$ and $AS$, respectively, and using the constraint in~\eqref{eq:SAXASE} gives
\begin{equation}\label{eq:BSAAS}  B = -(S^\top A^2 S)\Gamma S^\top A^2 S.
\end{equation}
The equation~\eqref{eq:BSAAS} is equivalent to solving in $\Gamma$ the following system
\begin{eqnarray}
(S^\top A^2 S)Y & = & -B \label{eq:BSAAS2lin1}\\
(S^\top A^2 S) \Gamma^\top & = & Y^\top. \label{eq:BSAAS2lin2}
\end{eqnarray}
The solution to~\eqref{eq:BSAAS2lin1} is given by any $Y$ such that
\begin{eqnarray}
\Range{Y} &\subset & -(S^\top A^2 S)^\dagger B + \Null{S^\top A^2 S} \nonumber \\
& \overset{\text{Lemma~\ref{lem:09709s} }}{ \subset} & -(S^\top A^2 S)^\dagger B + \Null{A S}, \label{eq:YSAA^TSsol}
\end{eqnarray}
where we applied Lemma~\ref{lem:09709s} with $G = I$ and $W = A S.$
The solution to~\eqref{eq:BSAAS2lin2} is given by any $\Gamma$ that satisfies
\begin{eqnarray}
\Range{\Gamma^\top} & \subset &  (S^\top A^2S)^\dagger Y^\top + \Null{S^\top A^2 S} \nonumber \\
& \overset{\text{Lemma~\ref{lem:09709s} } }{ \subset} & 
(S^\top A^2 S)^\dagger Y^\top + \Null{AS}.
\end{eqnarray}
Transposing the above, substituting~\eqref{eq:YSAA^TSsol}, left and right multiplying by $A S$ and $S^\top A$ respectively gives
\begin{eqnarray}
\Range{A S\Gamma S^\top A} &\subset &   A S\left(\Null{A S}-(S^\top A^2 S)^\dagger B\right)(S^\top A^2 S)^\dagger  S^\top A \\
& & + A S\Null{AS}^\top  S^\top A \nonumber \\
&= & 
-A S(S^\top A^2 S)^\dagger B (S^\top A^2 S)^\dagger  S^\top A,
\label{eq:gammasolSAXASE}
\end{eqnarray}
where in the last step we used the fact that
$ (A S\Null{AS}^\top  S^\top A)^\top = AS\Null{AS} S^\top A = 0.$
 Inserting~\eqref{eq:gammasolSAXASE} into~\eqref{eq:LagSAXASEsoltemp} gives
$E =  A S (S^\top A^2 S)^{\dagger} B (S^\top A^2 S)^\dagger S^\top A.$
Substituting in the definition of $E$ and $B$ we have~\eqref{eq:SAXASsol}.

For the dual problem, using~\eqref{eq:LagSAXASEsoltemp} and substituting~\eqref{eq:BAXA} into~\eqref{eq:Lagrangianeq:SAXASE}
 gives
\begin{eqnarray*}
  L(E,\Gamma) & = & \frac{1}{2}\norm{E}^2 - \norm{E}^2 - \dotprod{\Gamma, S^\top A(A^{\dagger}-X_k) AS} \\
 &= & -\frac{1}{2}\norm{E}^2  - \dotprod{A S\Gamma S^\top A, X_k-A^\dagger} \\
 & \overset{\eqref{eq:LagSAXASEsoltemp} }{=} & -\frac{1}{2}\norm{E}^2  + \dotprod{E, X_k-A^\dagger}  \pm \frac{1}{2} \norm{ X_k-A^{\dagger}}^2\\
 & =& -\frac{1}{2} \norm{E - (A^{\dagger}-X_k) }^2 + \frac{1}{2} \norm{X_k -A^{\dagger}}^2.
\end{eqnarray*}
  Substituting $E = X-X_k$, maximizing in $\Gamma$ and minimizing in $X$ while observing the constraint~\eqref{eq:LagSAXASEsoltemp}, we arrive at~\eqref{eq:SAXASdual}.
  
 Furthermore,  substituting  $X_k =0$ and $S =I$ in~\eqref{eq:SAXAS} and~\eqref{eq:SAXASdual} gives~\eqref{eq:AXA} and~\eqref{eq:AXAdual}, respectively, thus~\eqref{eq:AXA} and~\eqref{eq:AXAdual} are indeed equivalent dual formulations.  
 Finally, substituting $X_k =0$ and $S =I$ in~\eqref{eq:SAXASsol} and using properties P1 and P2, it is not hard to see that~\eqref{eq:SAXASsol} is equal to $A^\dagger,$ and thus~\eqref{eq:AXA} and~\eqref{eq:AXAdual} are indeed alternative characterizations of the pseudoinverse.  
 \hfill\qed
\end{proof}
  
One of the insights given by the dual formulation~\eqref{eq:SAXASdual} is that the resulting method is monotonic, that is, the
error $\norm{X_{k+1} -A^\dagger}$ must be a decreasing sequence. Inspired on the constraint in~\eqref{eq:SAXAS}, we refer to the class of methods defined by~\eqref{eq:SAXAS} as the \texttt{SAXAS} methods.

\subsection{Convergence}
\label{sec:conv}
Proving the convergence of the iterates~\eqref{eq:SAXASsol} follows the same machinery as the convergence proof in Section~\ref{subsec:ATAXconv}. But different from Section~\ref{subsec:ATAXconv} the resulting convergence rate $\rho$ may be equal to one $\rho =1$. We determine discrete distributions for $S$ such that $\rho <1$ in Section~\ref{sec:rateSAXAS}.

The first step of proving convergence is the following invariance result.
\begin{lemma}\label{lem:invariantAXA} 
 Let $A,W\in\R^{n \times n}$ be  symmetric matrices. If $X_0  = A W A$ then for each $k\geq 0$ there exists  matrix $Q_k\in \R^{n\times m}$ such that the iterates~\eqref{eq:SAXASsol} satisfy $X_k-A^\dagger = A Q_k A.$
\end{lemma}
\begin{proof}
Using induction and the constraint in~\eqref{eq:SAXASdual} we have that $X_{k+1} = A W_{k+1} A$ where $W_{k+1}= W_k+ S\Gamma S^\top$. Furthermore, from the constraint in~\eqref{eq:SAXASdual}, we have that there exists $\Gamma$ such that $A^\dagger = A \Gamma A.$ Thus $X_{k+1} -A^{\dagger} = A Q_{k+1}A$ with $Q_{k+1}= W_{k+1} - \Gamma.$\hfill\qed
\end{proof}

\begin{theorem} Let $A,W\in\R^{n \times n}$ be  symmetric matrices. If  $X_0  = A W A$ then the iterates~\eqref{eq:SATAXsol} converge according to
\begin{equation} \label{eq:AXAL2conv}
 \E{\norm{X_{k} - A^\dagger}^2 } \quad \leq \quad
   \rho ^{k}  \norm{X_{0}-A^\dagger}^2,
   \end{equation}
   where 
   \begin{equation} \label{eq:rhoAXA}
 \rho \quad\eqdef \quad 1- \inf_{\colvec{R = A Q A, Q \in \R^{n\times n} \\ \norm{R}^2 =1}} \dotprod{\E{ZRZ},R},
 \end{equation}
 and
 \begin{equation}
 Z  \eqdef  A S (S^\top A^2 S)^{\dagger} S^\top A.
 \end{equation}
\end{theorem}
\begin{proof}
Let $R_k = X_k - A^\dagger$.
Using 
\begin{equation}  \label{eq:BAXA}
 S^\top(A-AX_kA)S \overset{ (P1) }{=} S^\top A(A^{\dagger}-X_k) AS = S^\top A R_k AS,\end{equation}
and subtracting $A^\dagger$ from both sides of~\eqref{eq:SAXASsol} gives 
\begin{equation}\label{eq:SAXASR}
 R_{k+1} = R_k -  Z R_k Z.
\end{equation}
Applying the properties of the pseudoinverse, it can be  shown that $Z$ is a projection matrix, whence $Z^2 = Z$. 
Taking norms  and expectation conditioned on $R_k$ on both sides gives
\begin{eqnarray}
 \E{\norm{R_{k+1}}^2  \, | \, R_k}& = &
  \E{ \dotprod{R_k -  Z R_k Z, R_k - Z R_k Z} \, | \, R_k} \nonumber \\
  &\overset{Z \mbox{ \small is a proj.}}{=}&  \norm{R_k}^2 - \dotprod{ \E{Z R_k Z},R_k}.\label{eq:SAXASconv1}
\end{eqnarray}
By Lemma~\ref{lem:invariantAXA} and~\eqref{eq:SAXASconv1} we have that
\begin{eqnarray}
 \E{\norm{R_{k+1}}^2  \, | \, R_k}
  & \overset{\eqref{eq:SAXASconv1}}{=}&  \norm{R_k}^2 - \dotprod{ \E{Z\frac{R_k}{\norm{R_k}}Z},\frac{R_k}{\norm{R_k}}} \norm{R_k}^2. \nonumber \\
  & \overset{\text{Lemma~\ref{lem:invariantAXA}}+\eqref{eq:rhoAXA}}{\leq}&
  \rho  \norm{R_k}^2.
\end{eqnarray}
It remains to take  expectations again, apply the tower property, and unroll the recurrence.\hfill \qed
\end{proof}

The method described in \eqref{eq:SAXAS} is particularly well suited to calculating an approximation to the pseudoinverse of symmetric matrices, since symmetry is preserved by the method.
\begin{lemma} [Symmetry invariance] If $X_0 = X_0^\top$ and $A = A^\top$ then the iterates~\eqref{eq:SAXASsol} are symmetric.
\end{lemma} 
\begin{proof}
The constraint in~\eqref{eq:SAXASdual} and induction shows that $X_k = X_k^\top$ holds for any $k.\hfill\qed$
\end{proof}



\subsection{The rate of convergence} 
\label{sec:rateSAXAS}

It is not immediately obvious that \eqref{eq:rhoSATAX} is a valid rate. That is, is it the case that $0 \leq \rho \leq 1$? We give an affirmative answer to this  in Lemma~\ref{lem:rhoAXArealrate}. 
Subsequently, in Lemma~\ref{lem:sufrhonotzero} we establish necessary and sufficient  conditions on discrete distribution $\cal D$ to characterize when $\rho <1$. Consequently, under these conditions a linear convergence rate is guaranteed. 

To establish the next results we make use of vectorization and the Kronecker product so that we can leverage on classic results in linear algebra. For convenience, we state several well known properties and equalities involving Kronecker products in the following lemma.
But first, the Kronecker product of matrices $A\in \R^{m\times n}$ and $B\in \R^{p\times q}$ is defined as
\begin{equation}(A\otimes B)_{p(r-1)+i, q(s-1)+j} = a_{rs} b_{ij}.\label{eq:defkron}\end{equation}
Let $\vec{A} \in \R^{nm}$ denote the vector obtained by stacking the columns of the matrix $A$ on top of one another.
\begin{lemma}[Properties Kronecher products]\label{lem:kron}
For matrices $A, B$ and $C$ of compatible dimensions we have that
\begin{enumerate}
\item \label{it:kronprod}  $\overrightarrow{ABC}= (C^\top \otimes A) \vec{B},$
\item \label{it:krontrans} $(C\otimes B)^\top = C^\top \otimes B^\top.$
\item \label{it:kronpsd} If $A$ and $B$ are symmetric positive semidefinite then $B\otimes A$ is symmetric positive semidefinite.
\item \label{it:kronexp} Since both vectorization and expectation are linear operators, if $Z$ is a random matrix then $\overrightarrow{\E{Z}} = \E{\vec{Z}}.$
\end{enumerate}
\end{lemma}


\begin{lemma} \label{lem:rhoAXArealrate}The rate~\eqref{eq:rhoAXA} satisfies
$0 \leq \rho \leq 1.$
Furthermore, if 
\begin{equation} \label{eq:AXArhonot0}
\{R \, : \, \E{ZRZ} =0\} \subset \{ R \, : \, ARA =0\}, 
\end{equation}
then 
\begin{equation}
\rho \leq 1 -\lambda_{\min}^+(\E{Z \otimes Z}) <1,
\label{eq:rhoAXAlambda}\end{equation}
and the iterates~\eqref{eq:SATAXsol} converge.
\end{lemma}
\begin{proof}
Since $Z$ is positive semidefinite we have that 
\[\dotprod{ZRZ,R} = \Tr{R^\top ZRZ} = \Tr{Z^{1/2}R^\top ZRZ^{1/2}} \geq 0. \]
 Taking expectation in the above gives that $\rho \leq 1.$ Furthermore, since $Z$ is a projection matrix,
\begin{eqnarray*}
\dotprod{ZRZ,R} &=& \Tr{R^\top ZRZ}\\
& \leq & \Tr{R^\top ZR}  \underbrace{\lambda_{\max}(Z)}_{=1} \\
& = & \Tr{ ZRR^\top} \\
& \leq & \Tr{RR^\top}  \lambda_{\max}(Z) = \norm{R}^2.
\end{eqnarray*}
Dividing by $\norm{R}^2$ and taking expectation over $Z$  gives
\begin{equation}\frac{\dotprod{\E{ZRZ},R}}{\norm{R}^2} \leq 1. 
\label{eq:adj9823r} \end{equation}
Thus, for any $R \neq 0$, we have that
\[\rho \overset{\eqref{eq:rhoAXA}}{ \geq} 1-\frac{\dotprod{\E{ZRZ},R}}{\norm{R}^2} \overset{\eqref{eq:adj9823r}}{ \geq} 0,  \]
which concludes the proof that $0 \leq \rho \leq 1.$

After vectorizing and using item~\ref{it:kronprod} of Lemma~\ref{lem:kron}, the condition~\eqref{eq:AXArhonot0} is equivalent to
\begin{equation} \label{eq:asd98ja9}
\{\vec{R} \, : \, \E{Z\otimes Z}\vec{R} =0\} =\Null{\E{Z\otimes Z}} \overset{\eqref{eq:AXArhonot0}}{ \subset} \{ \vec{R} \, : \, (A \otimes A) \vec{R} =0\} = \Null{A \otimes A}. \end{equation}
Since $Z$  is symmetric positive semidefinite,  item~\ref{it:kronpsd} of Lemma~\ref{lem:kron} states that the matrix $Z\otimes Z$, and consequently $\E{Z\otimes Z}$, are symmetric positive semidefinite.
Thus taking orthogonal complements in~\eqref{eq:asd98ja9} we have
\begin{equation} \Range{A \otimes A}  \subset \Range{\E{Z\otimes Z}}.
\label{eq:oipawdk38rja}
\end{equation}
Therefore, using vectorization we have
\begin{eqnarray}
\inf_{\colvec{R = A Q A, Q \in \R^{n\times n} \\ \norm{R}^2 =1}} \dotprod{\E{ZRZ},R} & =& 
\inf_{\colvec{R \in \Range{A \otimes A}\\ \norm{R}^2 =1}} \dotprod{\E{Z\otimes Z}\vec{R},\vec{R}}_2 \nonumber \\
& \overset{\eqref{eq:oipawdk38rja} }{\geq } & \inf_{\colvec{R \in \Range{\E{Z \otimes Z}}\\ \norm{R}^2 =1}} \dotprod{\E{Z\otimes Z}\vec{R},\vec{R}}_2 \nonumber \\
& \overset{\eqref{eq:lambdamindef}}{=}& \lambda_{\min}^+ (\E{Z\otimes Z}) >1, \label{eq:randsau9}
\end{eqnarray}
where we have used that for any $G$ positive semi-definite we have
\begin{equation} \label{eq:lambdamindef} \lambda_{\min}^+(G) = \inf_{\colvec{x \in \Null{G}^\perp \\ \norm{x}_2 =1}} \dotprod{Gx,x}.
\end{equation}
Combining~\eqref{eq:randsau9} with~\eqref{eq:rhoAXA} gives the desired result~\eqref{eq:rhoAXAlambda}.
\hfill \qed
\end{proof}

\subsubsection{Characterization of $\rho <1$ for discrete distributions}
The following lemma gives a practical characterization of the condition~\eqref{eq:AXArhonot0} for discrete distributions. 

\begin{lemma}\label{lem:sufrhonotzero}
Let $S$ be a random matrix with a discrete distribution such that 
$\mathbb{P}(S=S_i) = p_i >0,$ where $S_i \in \R^{n \times q_i}$ for $i =1,\ldots, r.$ Let
\begin{equation}\label{eq:Sotimesstack} \mathbb{S} \eqdef\left(\colvec{S_1^\top \otimes S_1^\top\\ \vdots \\ S_r^\top \otimes S_r^\top }\right) \in \R^{\sum_{i=1}^rq_i^2 \times n}. \end{equation}
Then the iterates~\eqref{eq:SAXASsol} converge according to Theorem~\ref{theo:SAXAS} with a rate $\rho<1$  if
\begin{equation}\label{eq:sufrhonotzero}
\bigcap_{i=1}^r \{R \,: \, S_i^\top ARA S_i =0\} \subset \{R\, : \, ARA =0\}.
\end{equation}
Equivalently, condition~\eqref{eq:sufrhonotzero} holds if and only if
\begin{equation} \label{eq:sufrhonotzerovec}
 \Null{\mathbb{S} \, (A \otimes A)}\subset  \Null{A \otimes A}.\end{equation}
\end{lemma}
\begin{proof}
We show that~\eqref{eq:AXArhonot0} and~\eqref{eq:sufrhonotzero} are equivalent, therefore convergence of the iterates~\eqref{eq:SAXASsol} with $\rho < 1$ is guaranteed by Lemma~\ref{lem:rhoAXArealrate}. 
First, note once more that $\Null{\E{Z \otimes Z}} = \{\vec{R} \, : \,  \E{ZRZ} =0\}.$
Let $Z_i \eqdef A S_i (S_i^\top A^2 S_i)^{\dagger} S_i^\top A$ and note that $Z_i$ is a symmetric positive semidefinite matrix. Using the distribution of $S$ we have that $\vec{R} \in \Null{\E{Z \otimes Z}}$ is equivalent to
\begin{eqnarray}
 \E{ZRZ}  &=& 
 \sum_{i=1}^r p_i  Z_i R Z_i =0.
\end{eqnarray}
Since $Z_i$ is symmetric positive semidefinite by Lemma~\ref{lem:kron} items~\ref{it:kronpsd} and~\ref{it:kronexp} we have that $\E{Z_i \otimes Z_i}$ is positive definite, consequently
\begin{eqnarray}
\Null{\E{Z \otimes Z}} &=& \{\vec{R} \, : \,  \sum_{i=1}^r p_i (Z_i \otimes  Z_i) \vec{R} =0 \} \nonumber\\
&=& \{\vec{R} \, : \,   (Z_i \otimes  Z_i) \vec{R} =0, \quad \mbox{for } i=1,\ldots, r \} \nonumber \\
&= & \bigcap_{i=1}^r \Null{Z_i \otimes  Z_i}. \label{eq:lasd98j23}
\end{eqnarray}
Fix an index $i \in \{1,\ldots, r\}.$
The remainder of the proof is now dedicated to showing that 
$ \Null{Z_i \otimes  Z_i} = \{\vec{R} \, : \, S_i^\top A RA S_i =0 \}.$
To this end, we collect some facts. Given that 
\[\Null{ (S_i^\top A^2S_i)^\dagger} \overset{\eqref{it:consprop4}}{=}  \Null{S_i^\top A^2S_i} \overset{\text{Lemma~\ref{lem:09709s}} }{=} \Null{AS_i}, \quad \mbox{for } i=1,\ldots, r, \]
 we can apply Lemma~\ref{lem:09709s} once again with $G =(S_i^\top A^2S_i)^\dagger $ and $W = S_i^\top A$ which shows that 
\begin{equation}
\Null{Z_i} =\Null{A S_i (S_i^\top A^2 S_i)^{\dagger} S_i^\top A} \overset{\text{Lemma~\ref{lem:09709s} }}{=}  \Null{S_i^\top A}.
\label{eq:Zinullasdas}\end{equation}
Consequently
\begin{equation}
\Null{Z_i \otimes Z_i} =\{\vec{R} \, : \, Z_i R Z_i =0\} \overset{\eqref{eq:Zinullasdas}}{=} \{\vec{R} \, : \, S_i^\top A R Z_i =0\}
\overset{\eqref{eq:Zinullasdas}}{=} \{\vec{R} \, : \, S_i^\top A R A S_i =0\}
\label{eq:pjasdja89q}.\end{equation}
Finally
\[\Null{\E{Z \otimes Z}} \overset{\eqref{eq:lasd98j23}+\eqref{eq:pjasdja89q}}{=} \{ \vec{R} \, : \, S_i^\top A  R AS_i =0, \quad i=1,\ldots, r\} = \bigcap_{i=1}^r \{\vec{R} \, : \,S_i^\top A  R AS_i =0 \},\]
which proves that \eqref{eq:AXArhonot0} and \eqref{eq:sufrhonotzero} are equivalent.
Using vectorization, the condition~\eqref{eq:sufrhonotzero} can be rewritten as
$\{ v \, : \, ( S_i^\top \otimes S_i^\top ) (A \otimes A)v =0, \, \mbox{for } i =1,\ldots ,r \} \subset  \{ v\, : \, (A \otimes A)v =0\},$
which is clearly equivalent to~\eqref{eq:sufrhonotzerovec}.$\hfill\qed$
\end{proof}

Lemma~\ref{lem:sufrhonotzero} gives us a practical rule for designing a distribution for $S$ such that convergence is guaranteed. Given that $\Null{A \otimes A}$ is not known to us, the easiest way to ensure that~\eqref{eq:sufrhonotzerovec} holds is if we choose a distribution for $S$ such that $\mathbb{S}$ has a full column rank. Clearly~\eqref{eq:sufrhonotzero} holds when $S$ is a fixed invertible matrix with probability one, but this does not result in a practical method. In the next section we show how to construct $S$ so that $\mathbb{S}$ has a full column rank and results in a practical method.

\subsection{Discrete examples}

Based on the two sketching strategies presented in Section~\ref{subsec:SATAXexe}, we define two variants of the 
\texttt{SAXAS} method~\eqref{eq:SAXASsol}. Let the \texttt{SAXAS\_uni} and the \texttt{SAXAS\_ada} methods 
be the result of using a uniform $\tau$--batch sketching and an adaptive sketching with the \texttt{SAXAS} method, respectively. We found that these two variants work well in practice, as we show later on in Section~\ref{sec:numerics}.  
Though we observe in empirical experiments that the two variants of \texttt{SAXAS} converge in practice, it is hard to verify Lemma~\ref{lem:sufrhonotzero} and thus  prove convergence. 
So instead we introduce a new sketching very similar to the uniform $\tau$--batch sketching, but that allows us to easily prove convergence of the resulting method. 
\paragraph{$\tau$--batch sketching with replacement.}
 Let $S= I_{:v}$ where $v \in \{1,\ldots, n\}^\tau$ is an array  and $I_{:v} \in \R^{n \times \tau}$ is the column concatenation of the columns in the identity matrix $I$ indexed by $v.$ Furthermore, let $\mathbf{P}(S=I_{:v}) = p_v >0$ for each $v \in  \{1,\ldots, n\}^\tau.$
%
%


We refer to the \texttt{SAXAS} method with a $\tau$--batch sketching with replacement as the \texttt{SAXAS\_rep} method. 
As we will now show, under the condition that $\tau\geq 2$, the \texttt{SAXAS\_rep} method satisfies Lemma~\ref{lem:sufrhonotzero} and thus convergence of the \texttt{SAXAS\_rep} method is guaranteed. 
\paragraph{Convergence.} We will prove that \texttt{SAXAS\_rep} method converges by showing that 
the matrix $\mathbb{S}$ defined in~\eqref{eq:Sotimesstack} has full column rank, and thus according to Lemma~\ref{lem:sufrhonotzero} the iterates converge. First note that since the sampling is done over all $v \in  \{1,\ldots, n\}^\tau$, there are $n^\tau$ different sketching matrices. Thus $\mathbb{S} \in \R^{\tau n^\tau \times n^2}.$ 
 To prove that $\mathbb{S}$ has full column rank, we will show that for $\tau\geq 2$ that the row rank of $\mathbb{S}$ is $n^2$. Note that for $\tau =1$ the matrix $\mathbb{S}$ has $n$ rows, thus it is not possible for $\mathbb{S}$  to have full column rank. For simplicity, consider the case $\tau=2.$ Fix $i \in \{1,\ldots, n^2\}.$ We will now show that for the $i$th coordinate vector $e_i \in \R^{n^2}$, there exists $v\in \{1,\ldots, n\}^\tau$ such that $e_i$ is a row of $I_{:v}^\top \otimes I_{:v}^\top$, and consequently, $e_i$ is a row of $\mathbb{S}.$ First, for $v=(s, j)$ we have from the definition of Kronecker product~\eqref{eq:defkron} that 
 \begin{equation} (I_{:v}^\top \otimes I_{:v}^\top)_{2, n(s-1)+j} = [I_{:v}]_{1s} [I_{:v}]_{2j} =1.\label{eq:IsjIsjproduasd}\end{equation}
  Moreover, every other element on row $2$ of $I_{:v}^\top \otimes I_{:v}^\top$ is zero apart from the element in column $n(s-1)+j.$ Now note that the integer $i$ can be written as 
   \[  i = n\left\lfloor \frac{i}{n} \right\rfloor + \mod(i,n) =
  n\underbrace{\left(\left\lfloor \frac{i}{n} \right\rfloor -1\right)}_{=s-1} + \underbrace{\mod(i,n)+n}_{=j}.  \]
  By setting $s = \lfloor \frac{i}{n} \rfloor$ and $j = \mod(i,n)+n$, we have from the above that $n(s-1)+j =i$. Though there is problem when $\lfloor \frac{i}{n} \rfloor =0,$ since $s$ cannot be zero. To remedy this, consider the indices  
 \[ s= \begin{cases} 1 & \quad  \mbox{if } i <n \\
 \left\lfloor \frac{i}{n} \right\rfloor & \quad  \mbox{if } i \geq n,
  \end{cases} 
\qquad \mbox{and}  \qquad 
j= \begin{cases} i & \quad  \mbox{if } i <n \\
 \text{mod}(i,n) +n & \quad  \mbox{if } i \geq n.
  \end{cases}
  \]
  With $v=(s,j)$ we now have that the 2nd row of the matrix in~\eqref{eq:IsjIsjproduasd} is the $i$th unit coordinate vector in $\R^{n^2}.$  Consequently $\mathbb{S}$ has row rank $n^2$ and the \texttt{SAXAS\_rep} method converges.

 
\section{Projections and Full Rank Matrices} 
 
In this section we comment on calculating approximate projections onto the range space of a given matrix, and on certain specifics related to calculating the pseudoinverse of a full rank matrix.
 
\subsection{Calculating approximate range space projections}

 With very similar methods, we can calculate an approximate projection operator onto the range space of $A$. Note that $AA^\dagger$ projects onto $\Range{A}$ as can be seen by $(P1)$. But rather than calculate $A^\dagger$ and then left multiply by $A$, it is more efficient to calculate $A A^\dagger$ directly. For this, let $P \eqdef A A^\dagger$ and note that from  the identities $AA^\dagger A =A$ and $A^\top AA^\dagger = A^\top$ we have that $P$ satisfies
\begin{enumerate}
\item \label{it:projprop1}  $PA = A$
\item \label{it:projprop2}  $A^\top P = A^\top.$
\end{enumerate}
 We can design a sketch and project method based on either property. For instance, based on item~\ref{it:projprop1} we have the method
\begin{equation}\label{eq:projiter}
X_{k+1} \eqdef \arg \min \frac{1}{2}\norm{X-X_k}^2,\quad  \mbox{subject to} \quad PAS   = AS.
\end{equation}
The advantage of this approach, over calculating $A^\dagger$ separately, is a resulting faster method. Indeed, if we were to carry out the analysis of this method, following analogous steps to the convergence in Section~\ref{subsec:ATAXconv}, and together with a conveniently chosen probability distribution based on Lemma~\ref{lem:convprob}, the iterates~\eqref{eq:projiter} would converge according to
\begin{equation} \label{eq:EXprojA}
\E{\norm{X_{k+1}-P}^2} =\left(1- \frac{\lambda_{\min}^+ (\mathbb{S}^\top  A^\top A \mathbb{S})}{\Tr{\mathbb{S}^\top  A^\top A  \mathbb{S}}}\right) \E{\norm{X_k-P}^2}. \end{equation}
  Since the rate is proportional to a scaled condition number with fewer powers of $A$ as compared to our previous convergence results~\eqref{eq:SATAXExpnorm}, the method~\eqref{eq:projiter} is less sensitive to ill conditioning in the matrix $A$. 
  
Such a method would be useful in a solving linearly constrained optimization problems~\cite{GoulHribNoce01,Calamai1987} which often require projecting the gradient onto the range space of system matrix. In particular, in a iteration of a Newton-CG framework~\cite{Dembo1982,Gondzio2013}, one needs only inexact solutions to a quadratic optimization problem with linear constraints. A method based on~\eqref{eq:projiter} can be used to calculate a projection operator to within the precision required by the Newton-CG framework, and thus save on the computational effort of calculating the exact projection matrix.
  
\subsection{Pseudoinverse of full rank matrices}
\label{sec:fullrank}

In the special case when $A$ has full rank, there are two alternative sketch-and-project methods that are more effective than our generic method.  In particular, when $A$ has full row rank ($m \leq n$) then there exists $X$ such that $AX =I$, furthermore, $A A^\dagger = I$. In this case, we have that
\begin{equation}
A^\dagger = \arg \min \norm{X}_F^2,\quad  \mbox{subject to} \quad
AX =I.
\end{equation}  
Applying a sketching and projecting strategy to the above gives
\begin{equation}\label{eq:SAX}
X^{k+1}  = \arg \min \norm{X-X^k}_F^2,\quad  \mbox{subject to} \quad
S^\top AX =S^\top .
\end{equation}
This method~\eqref{eq:SAX} was presented in~\cite{Gower2016} as a method for inverting matrices. The analysis in~\cite{Gower2016} still holds in this situation by using the techniques we presented in Section~\ref{sec:conv}. Again, the resulting rate of convergence of the method defined by~\eqref{eq:SAX} is less sensitive to ill conditioning in the matrix $A$, as can be seen in Theorem~6.2 in \cite{Gower2016}.

Alternatively, when $A$ has full column rank, then $A^\dagger A =I,$ and one should apply a sketching and projecting method using the equation $XA=I.$

Consequently the methods \texttt{SATAX}~\eqref{eq:SATAXsol} and \texttt{SAXAS}~\eqref{eq:SAXASsol} are better suited for calculating the pseudoinverse of rank deficient matrices, which is the focus of our experiments in the next section.
\section{Numerical Experiments}
\label{sec:numerics}
We now perform several numerical experiments comparing  two variants of the \texttt{SATAX} and the~\texttt{SAXAS} methods
to the Newton-Schulz method
\begin{equation}
X^{k+1} = 2X^k + X^kAX^k,
\end{equation}
 as introduced by Ben-Israel and Cohen~\cite{Ben-Israel1966b,Ben-Israel1965} for calculating the pseudoinverse matrix. 
 The Newton-Schulz method is guaranteed to converge as long as $\norm{I-X_0A}_2 <1$. Consequently, we set $X_0= \tfrac{1}{2} \tfrac{A^\top }{\norm{A}_F^2}$ for the Newton-Schulz method to guarantee its convergence. Furthermore, the Newton-Schulz method enjoys quadratic local convergence~\cite{Ben-Israel1966b,Ben-Israel1965}, in contrast to the randomized methods which are globally linearly convergent.
 Thus in theory the Newton-Schulz should be more effective at calculating a highly accurate approximation to the pseudoinverse as compared to the randomized methods, as we confirm in the next experiments.

\setlength\fboxsep{0pt}
All the code for the experiments is written in the \raisebox{-\mydepth}{\includegraphics[scale=0.035]{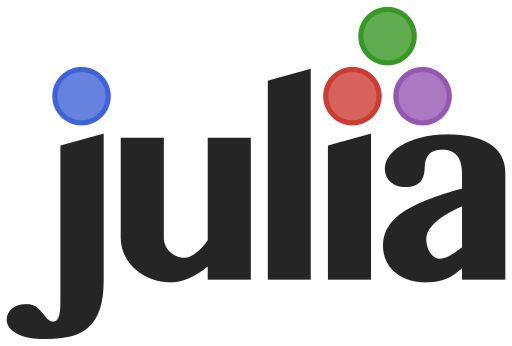}} programming language and can be downloaded from~\url{http://www.di.ens.fr/~rgower/} or~\url{https://github.com/gowerrobert/}.
 
In each figure presented below we plot the evolution of the residual $\norm{AXA-A}_F$ against time and flops of each method. 

\subsection{Nonsymmetric matrices}
 \label{subsec:num-nonsym}

In this section we compare the  \texttt{SATAX\_uni},  \texttt{SATAX\_ada} and Newton-Schulz methods presented earlier in Section~\ref{subsec:SATAXexe}.
 In setting the initial iterate $X_0$ for the \texttt{SATAX} methods, we know from Lemma~\ref{lem:invariant} and Theorem~\ref{theo:ATAX} that we need $X_0 = \alpha A^\top$ for some $\alpha \in \R$ to guarantee that the method converges. We choose $\alpha$ as
 \[ \alpha = \frac{\min\{n,m\}}{\norm{A}_F^2},\]
 which is an approximation to the solution of  
  \[\alpha^* = \arg \min \norm{A^\dagger - \alpha A^\top}_F^2,\]
 to which the exact solution is $\alpha^* = \left.\Rank{A} \right/ \norm{A}_F^2.$

To verify the performance of the methods, we test several rank deficient matrices from the UF sparse matrix collection\cite{Davis:2011}. In Figures~\ref{fig:lp_fit2d}, \ref{fig:lp_ken}, \ref{fig:NYPA-Maragal_6} and~\ref{fig:Meszaros-primagaz} we tested the three methods on the LPnetlib/lp\_fit2d, the LPnetlib/lp\_ken\_07, NYPA/Maragal\_6 and the Meszaros/primagaz problems, respectively. 
\begin{figure}\centering  
\includegraphics[width = 0.45\textwidth ]{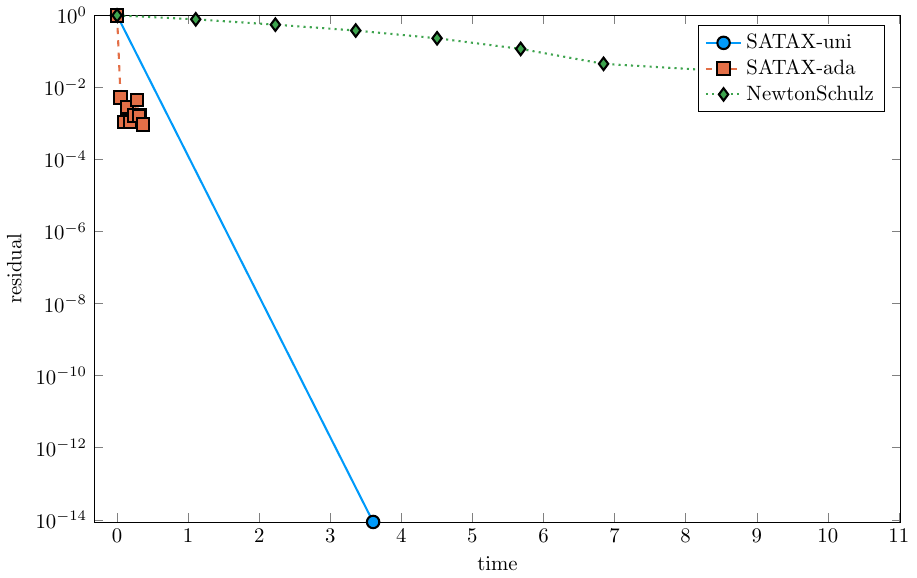} ~~~
\includegraphics[width = 0.45\textwidth ]{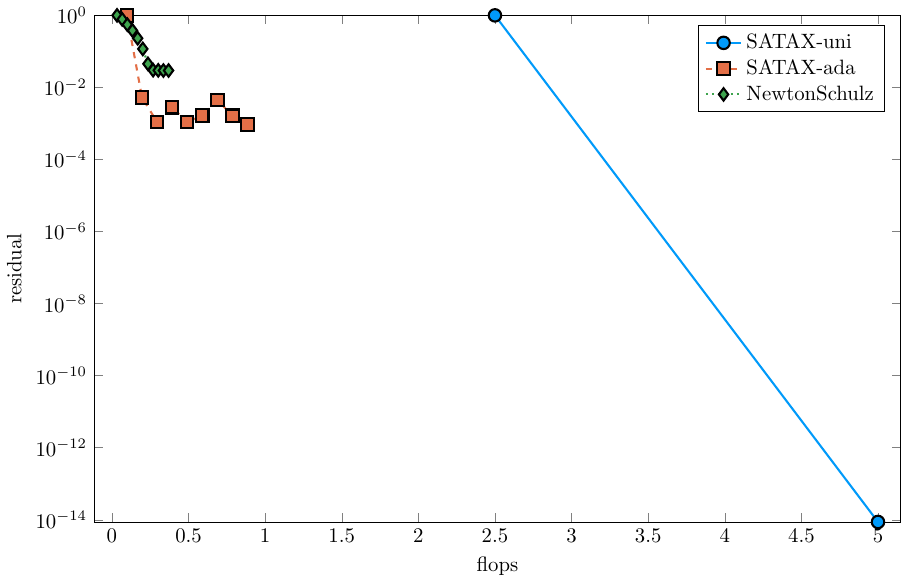}
\caption{LPnetlib/lp\_fit2d $(m;n) = (10,524 ; 25)$. }
\label{fig:lp_fit2d}
\end{figure}

\begin{figure}\centering  
\includegraphics[width = 0.45\textwidth ]{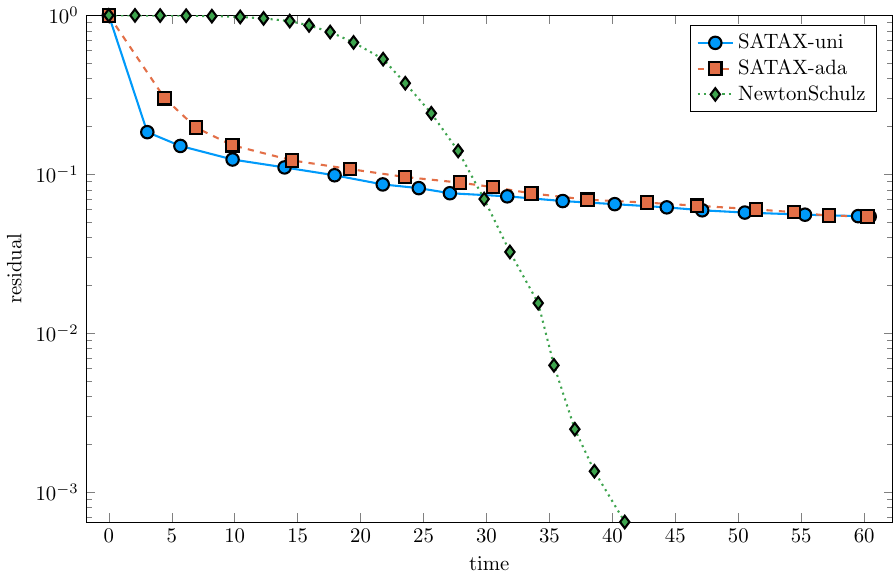}~~~
\includegraphics[width = 0.45\textwidth ]{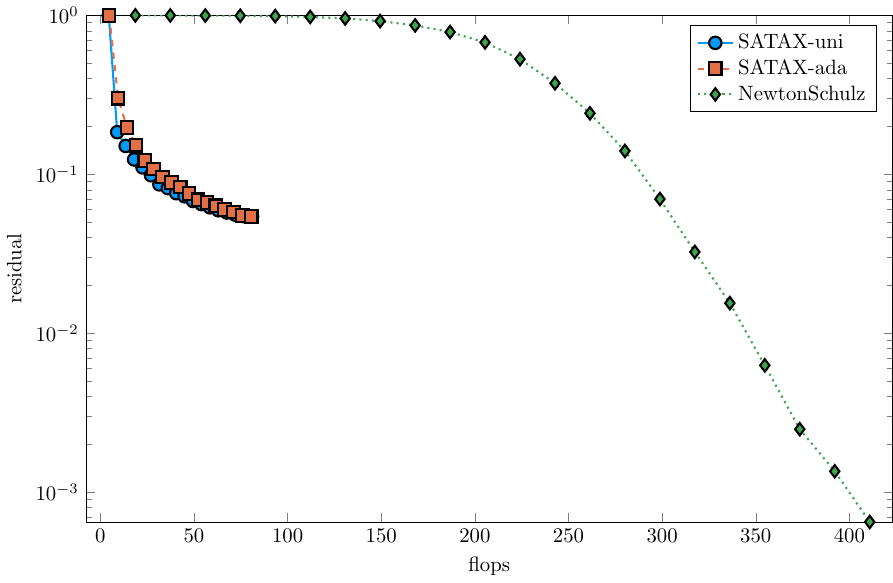}
\caption{LPnetlib/lp\_ken\_07 $(m;n) = (2,426 ; 3,602)$. }
\label{fig:lp_ken}
\end{figure}

\begin{figure}\centering  
\includegraphics[width = 0.45\textwidth ]{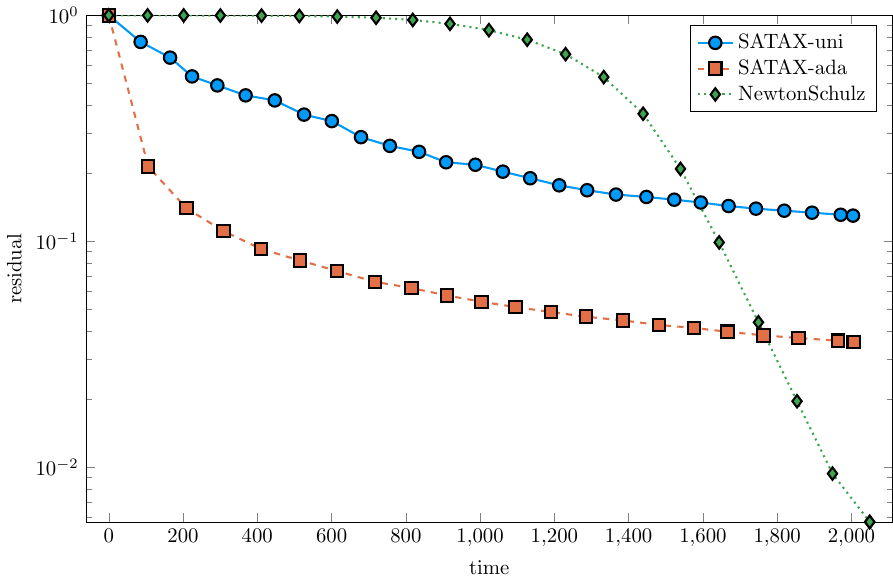}~~~
\includegraphics[width = 0.45\textwidth ]{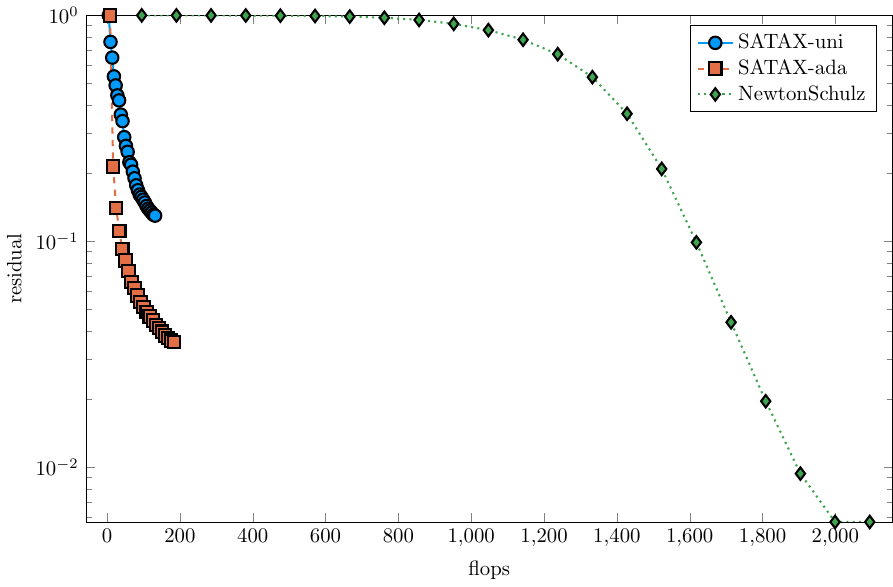}
\caption{NYPA/Maragal\_6 $(m;n) = (21,255	; 10,152	)$. }
\label{fig:NYPA-Maragal_6}
\end{figure}

\begin{figure}
\includegraphics[width = 0.45\textwidth ]{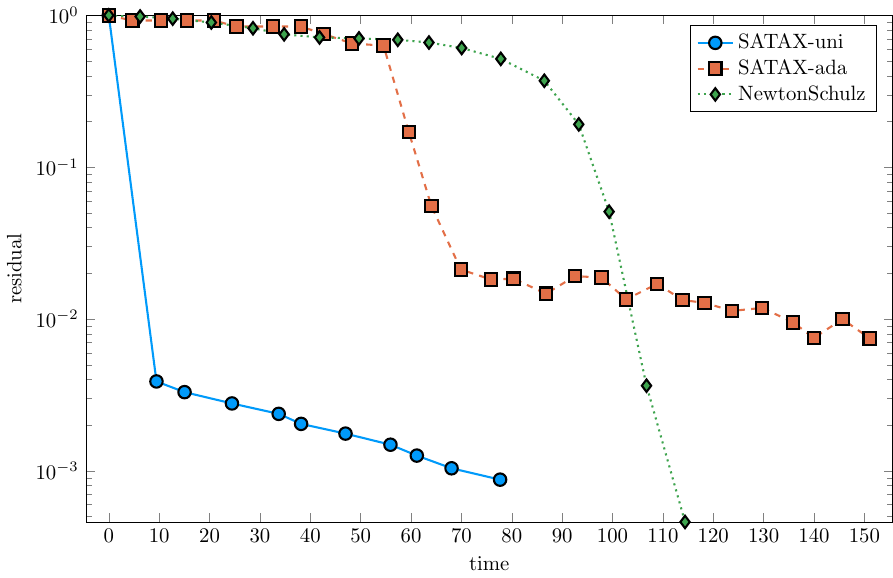}~~~
\includegraphics[width = 0.45\textwidth ]{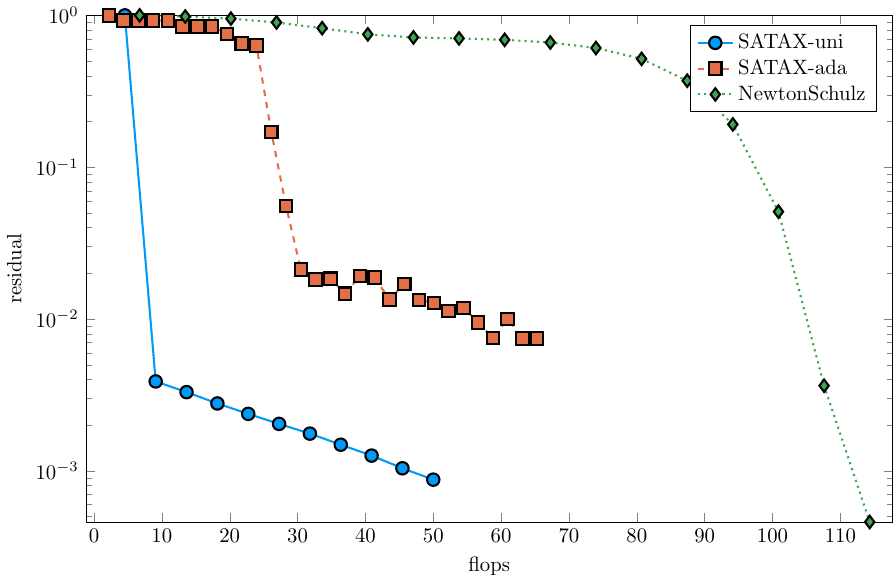}
\caption{Meszaros/primagaz $(m;n) = (1,554 ; 10,836)$ }
\label{fig:Meszaros-primagaz}
\end{figure}

From Figure~\ref{fig:lp_fit2d} we see that the \texttt{SATAX} methods are considerably faster at calculating the pseudoinverse on highly rectangular matrices ($n \ll m$ or $n \gg m$) as compared to the Newton-Schultz method. Indeed, by the time the Newton-Schultz method completes three iterations, the stochastic methods have already encountered a pseudoinverse within the desired accuracy. On the remaining problems in Figures~\ref{fig:lp_ken}, \ref{fig:NYPA-Maragal_6} and~\ref{fig:Meszaros-primagaz} the results are mixed, in that, the \texttt{SATAX} methods are very fast at encountering a rough approximation of the pseudoinverse with a residual between $10^{-1}$ and $10^{-3}$, but for reaching a lower residual the Newton-Schultz method proved to be the most efficient.

In calculating the approximate pseudoinverse of the the best rank $r=1000$ approximation to a random  $5000\times 2500$ Gaussian matrix the Newton-Schultz method outperforms the randomized methods in terms of time taken but is less efficient is terms of flops, see in Figure~\ref{fig:rand1}. We observed this same result holds for  Gaussian matrices with a range of different dimensions and different ranks. 
\begin{figure}
\includegraphics[width = 0.45\textwidth ]{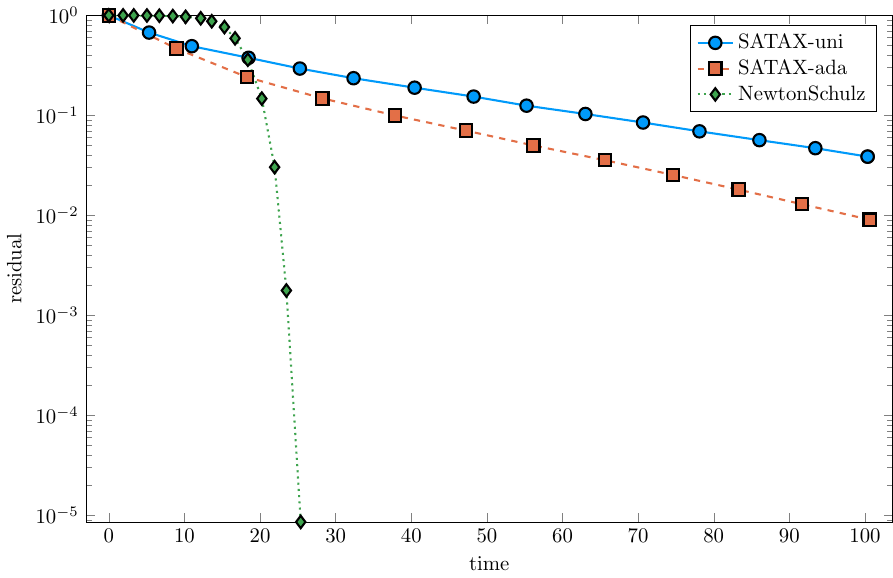}~~~
\includegraphics[width = 0.45\textwidth ]{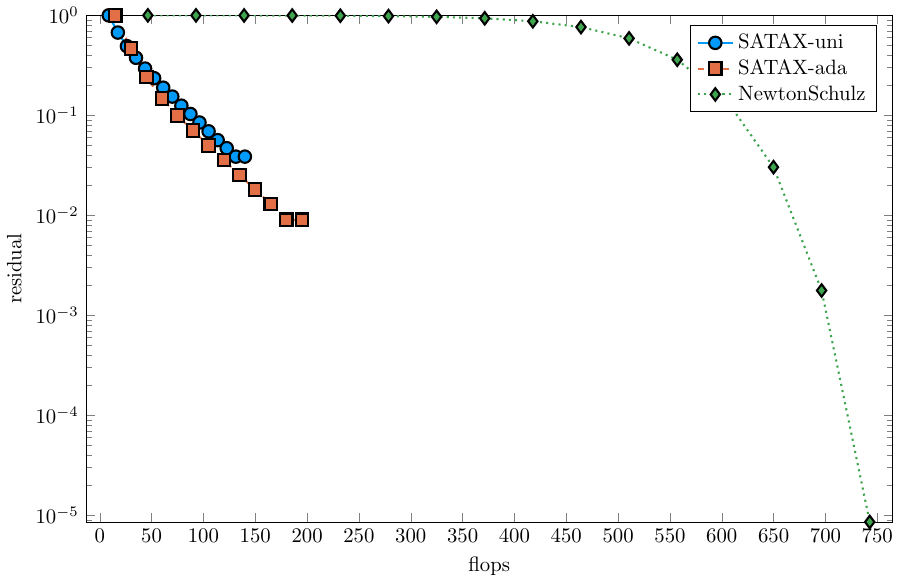}
\caption{The matrix $A$  is the best rank $1000$ approximation to a randomly generated  $5000\times 2500$ normal Gaussian matrix.}
\label{fig:rand1}
\end{figure}

The faster initial convergence of \texttt{SATAX} methods and the local quadratic convergence of the Newton-Schulz method can be combined to create an efficient method. To illustrate, we create a combined method named \texttt{NS-SATAX} where we use the \texttt{SATAX} method for the first few iterations before switching to the Newton-Schulz method, see Figure~\ref{fig:comb}. Through experiments we have identified that we should switch to the Newton-Schulz method after the \texttt{SATAX} method has performed one \emph{effective pass over the data}. In other words, we should switch methods after $t$ iterations such that $t$ times the cost of computing the sketched matrix $AS$ is equal to the cost of performing one full matrix-matrix product $A X$ where $X \in \R^{n \times m}$. Though this requires care, in particular, if $X_t$ is the last iteration of the \texttt{SATAX} method, then we need to ensure that $X_t$ satisfies the starting condition  $\norm{I-X_t A}_2 < 1$ of the Newton-Schulz method. For this we normalize the iterate $X_t$ according to $X_t \leftarrow \left.X_t \right/ \norm{X_t A}_F.$ This normalization is a heuristic and is not guaranteed to satisfy the Newton-Schulz starting condition. Despite this, it does work in practice as we can see in Figure~\ref{fig:comb} where the combined method~\texttt{NS-SATAX} outperforms the Newton-Schulz method during the entire execution.

\begin{figure}
\includegraphics[width = 0.45\textwidth ]{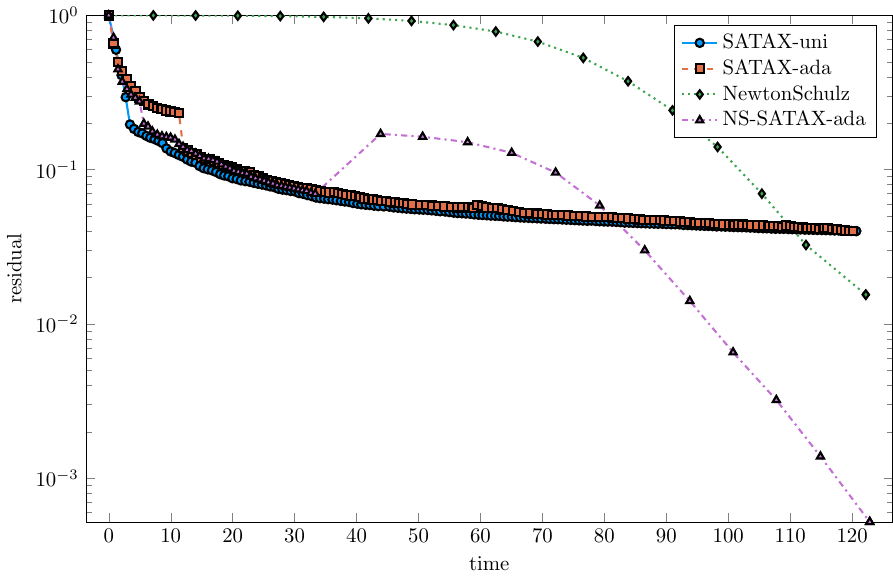}~~~
\includegraphics[width = 0.45\textwidth ]{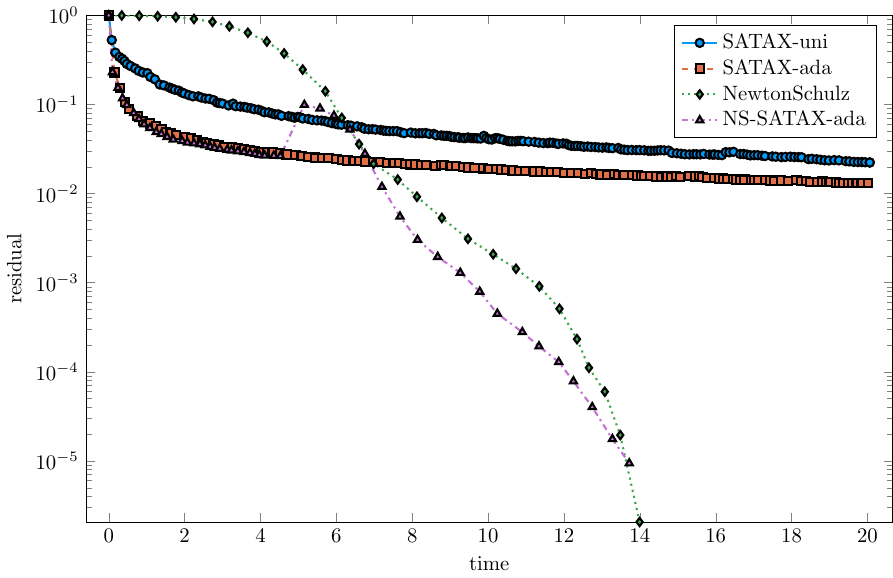}
\caption{Comparing the new combined method \texttt{NS-SATAX} to the \texttt{SATAX\_uni}, \texttt{SATAX\_ada} and the Newton-Schulz method on the LPnetlib/lp\_ken\_07 matrix (LEFT figure) and the NYPA/Maragal\_3 ( RIGHT figure).}
\label{fig:comb}
\end{figure}

\subsection{Symmetric matrices}

In this section we compare the  \texttt{SAXAS\_uni},  \texttt{SAXAS\_ada} and Newton-Schulz methods.  In setting the initial iterate $X_0$ for the \texttt{SAXAS} methods, we know from Lemma~\ref{lem:invariantAXA} and Theorem~\ref{theo:SAXAS} that we need $X_0 = \alpha A^2$ for some $\alpha \in \R$ to guarantee that the method converges. We choose $\alpha$ so that $\norm{X_0}_F^2 =1,$ that is $\alpha = \left. 1\right/ \norm{A}_F^2.$

To test the symmetric methods we used the Hessian matrix $A^\top A$ of the linear regression problem 
 \begin{equation}\label{eq:ridgeMatrix}
\min_{x\in \R^n} \frac{1}{2}\norm{Ax-b}_2^2
\end{equation}
using data from LIBSVM~\cite{Chang2011},  see Figure~\ref{fig:a9a},~\ref{fig:mushrooms},~\ref{fig:gisette_scale} and~\ref{fig:rcv1_train.binary}.
These experiments show that the two variants of the \texttt{SAXAS} method are much more efficient at calculating an approximate pseudoinverse as compared to the Newton-Schulz method, even for reaching a relative residual with a high precision  of around $10^{-6}$.
The only exception being the \texttt{rcv1\_train.binary} problem in Figure~\ref{fig:rcv1_train.binary}, where the~\texttt{SAXAS\_uni} and  \texttt{SAXAS\_ada} methods make very good progress in the first few iterations, but then struggle to bring the residual much below $10^{-2}.$ Again looking at Figure~\ref{fig:rcv1_train.binary}, the trend appears that the Newton-Schulz method will reach a lower precision than the the~\texttt{SAXAS\_uni} and  \texttt{SAXAS\_ada} after approximately $4000$ seconds, though we were not prepared to wait so long. We leave it as an observation that we could again get the best of both worlds by combining an initial execution of the \texttt{SAXAS} methods and later switching to the Newton-Schulz method as was done with the \texttt{SATAX} and Newton-Schulz method in the previous section.

\begin{figure}
\includegraphics[width = 0.45\textwidth ]{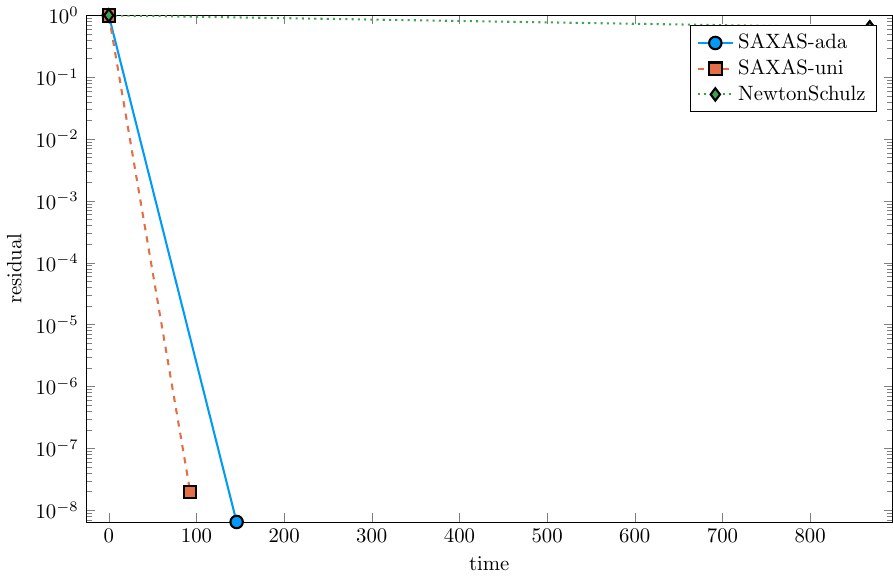}~~~
\includegraphics[width = 0.45\textwidth ]{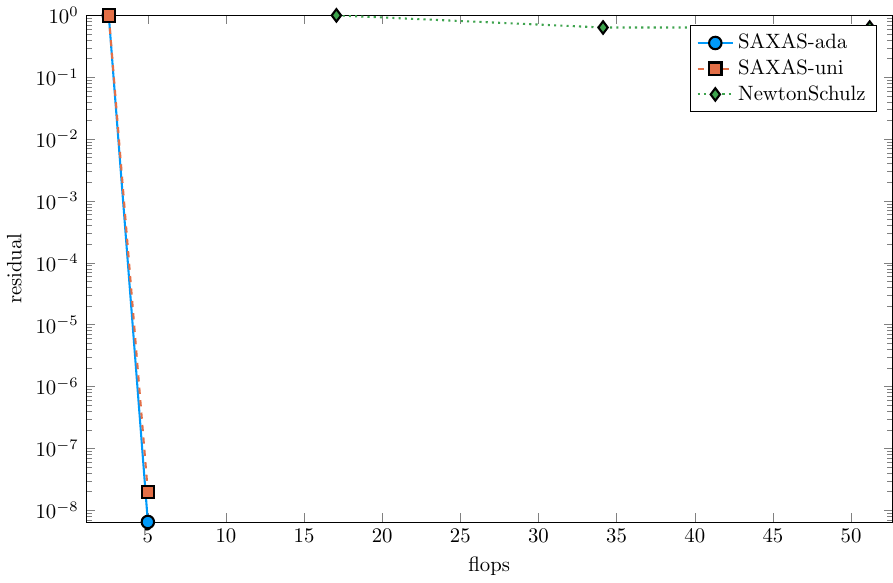}
\caption{a9a: $(m;n) = (123 ; 32,561)$.}
\label{fig:a9a}
\end{figure}

\begin{figure}
\includegraphics[width = 0.45\textwidth ]{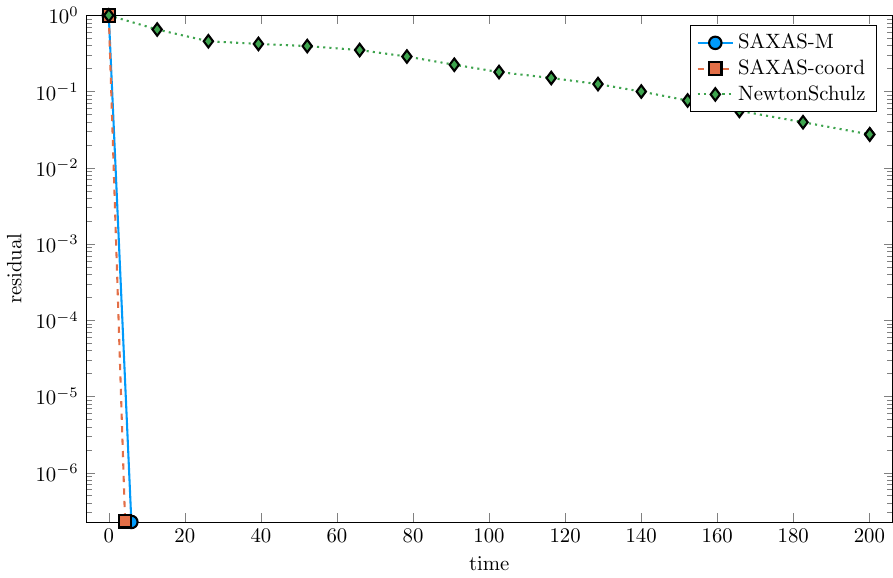}~~~
\includegraphics[width = 0.45\textwidth ]{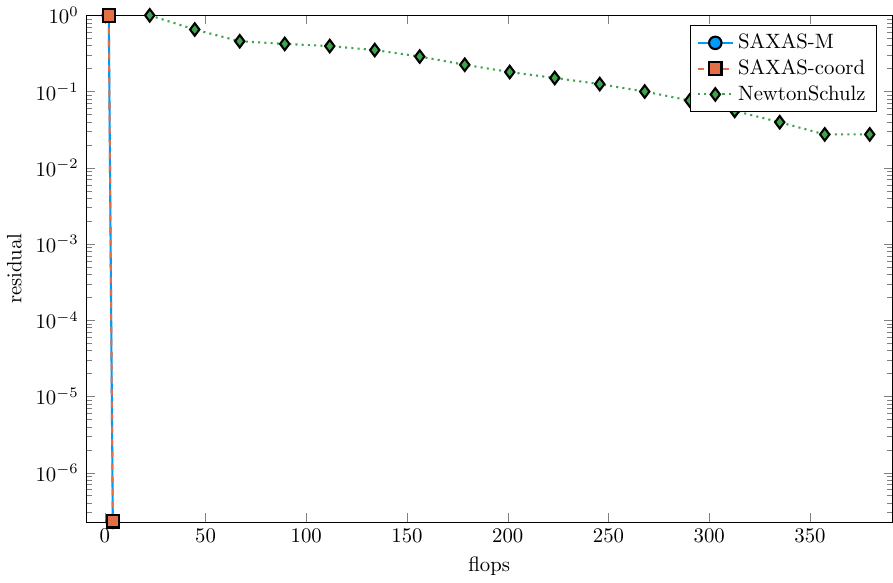}
\caption{mushrooms: $(m;n) = (8,124 ; 112)$ }
\label{fig:mushrooms}
\end{figure}

\begin{figure}
\includegraphics[width = 0.45\textwidth ]{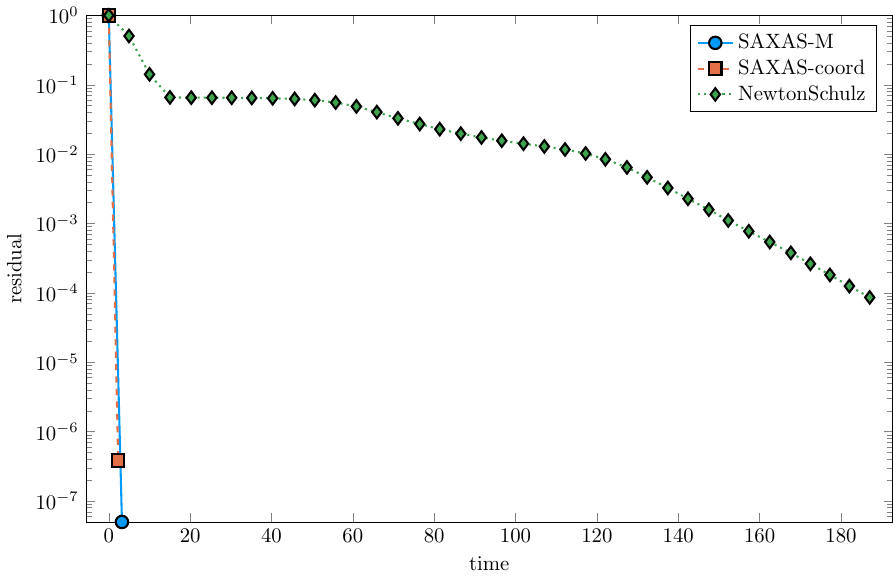}~~~
\includegraphics[width = 0.45\textwidth ]{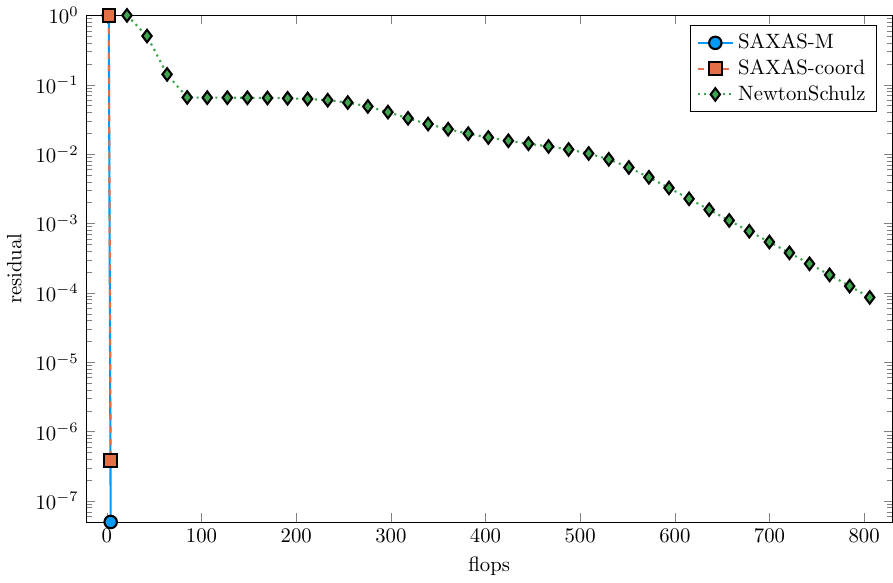}
\caption{gisette\_scale: $(m;n)=(6000; 5000)$ }
\label{fig:gisette_scale}
\end{figure}

\begin{figure}
\includegraphics[width = 0.45\textwidth ]{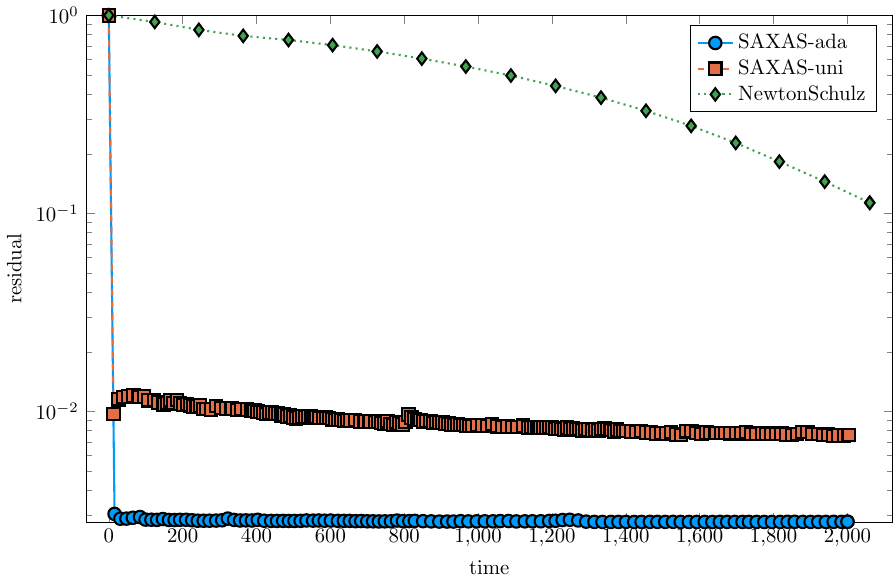}~~~
\includegraphics[width = 0.45\textwidth ]{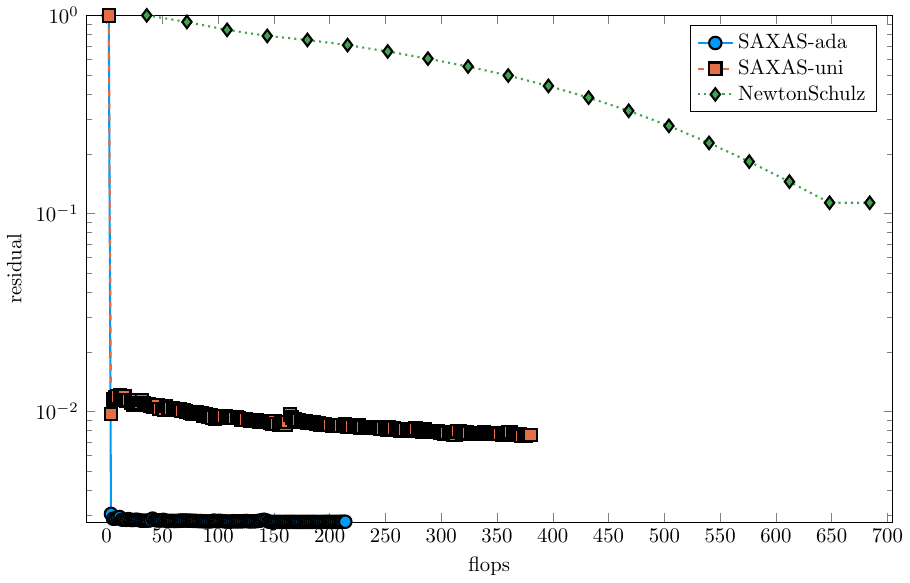}
\caption{rcv1\_train.binary: $(m;n)=(20,242; 47,236)$
 }
\label{fig:rcv1_train.binary}
\end{figure}

Again we found that the Newton-Schultz method was more efficient in calculating pseudoinverse of randomly generated Gaussian matrices $A$, where $A$ is the best rank $1000$ approximation to a matrix $G+G^\top$, where $G$ is a $5000 \times 5000$ random Gaussian matrix; see Figure~\ref{fig:rand}. 

\begin{figure}
\includegraphics[width = 0.45\textwidth ]{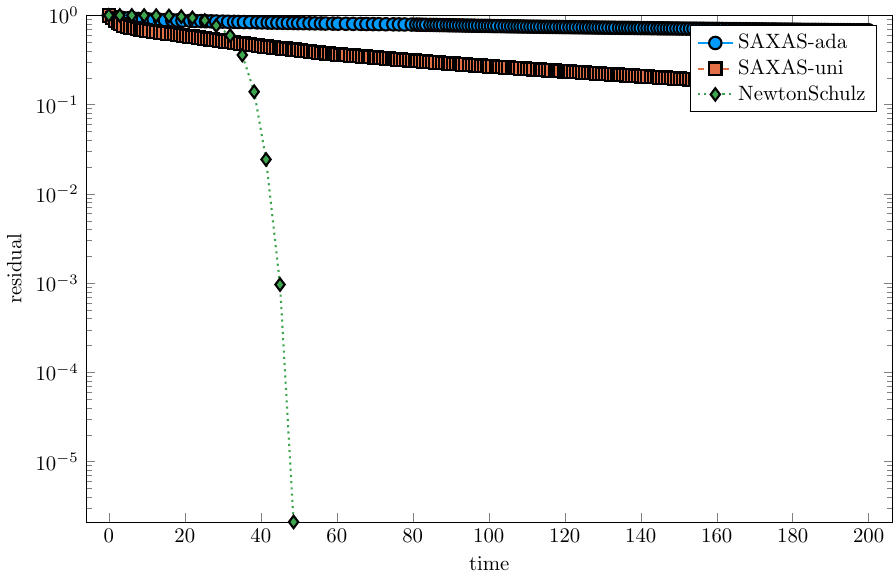}
\includegraphics[width = 0.45\textwidth ]{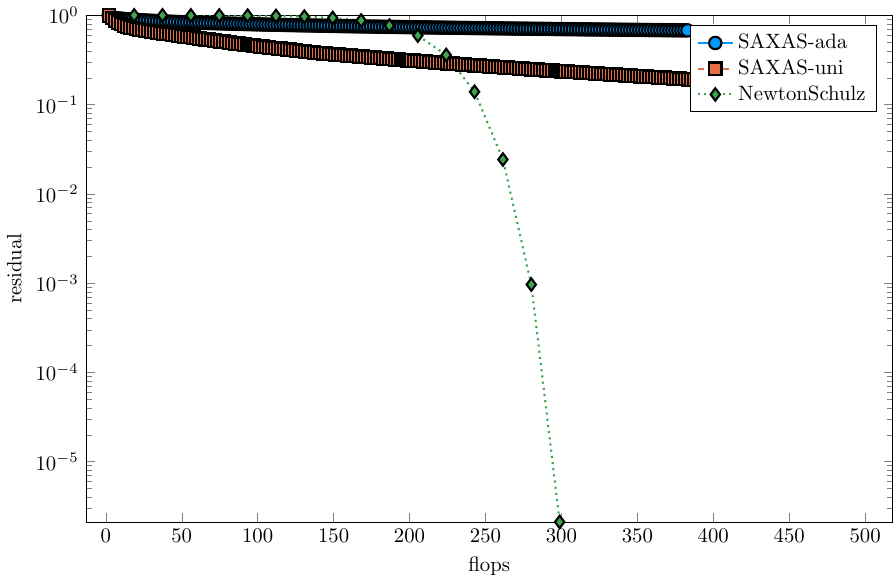}
\caption{The matrix $A$ is the best rank $10^3$ approximation to the matrix $G+G^\top$ where $G$ is a $5000 \times 5000$ random Gaussian matrix.}
\label{fig:rand}
\end{figure}

\section{Conclusions and Future Work}
We presented a new family of randomized methods for iteratively computing the pseudoinverse which are proven to converge linearly to the pseudoinverse matrix and, moreover, numeric experiments show that the new randomized methods are vastly superior at quickly obtaining an approximate pseudoinverse matrix. In such cases where an approximation of the pseudoinverse of a nonsymmetric matrix with a relative residual below $10^{-3}$ is required  then our experiments show that the Newton Schultz method is more effective as compared to our randomized methods. These observations inspired a combined method which we illustrated in~Figure~\ref{fig:comb} which has better overall performance than the Newton-Schulz method. Furthermore, we present new symmetric sketches used to design the \texttt{SAXAS} method.  
For future work, we have indicated how to design randomized methods for calculating approximate range space projections and pseudoinverse of full rank matrices.



{ 
\bibliographystyle{plain}
\bibliography{pseudo_inverse}

\begin{thebibliography}{10}

\bibitem{ADLER_ICASSP_2010}
Y.~Hel-Or A.~Adler and M.~Elad.
\newblock A weighted discriminative approach for image denoising with
  overcomplete representations.
\newblock {\em IEEE International Conference on Acoustics, Speech, and Signal
  Processing (ICASSP), Dallas, TX, 14-19 March}, March 2010.

\bibitem{Ben-Israel1965}
Adi Ben-Israel.
\newblock An iterative method for computing the generalized inverse of a
  matrix.
\newblock {\em Mathematics of Computation}, 19(91):452--455, 1965.

\bibitem{Ben-Israel1966b}
Adi Ben-Israel and Dan Cohen.
\newblock On iterative computation of generalized inverses and associated
  projections.
\newblock {\em SIAM Journal on Numerical Analysis}, 3(3):410--419, 1966.

\bibitem{Benzi1999}
Michele Benzi and Miroslav T{\r{u}}ma.
\newblock Comparative study of sparse approximate inverse preconditioners.
\newblock {\em Applied Numerical Mathematics}, 30(2--3):305--340, 1999.

\bibitem{Broyden1967}
C.~G. Broyden.
\newblock Quasi-{Newton} methods and their application to function
  minimisation.
\newblock {\em Mathematics of Computation}, 21(99):368--381, 1967.

\bibitem{Calamai1987}
Paul~H. Calamai and Jorge~J. Mor{\'e}.
\newblock Projected gradient methods for linearly constrained problems.
\newblock {\em Mathematical Programming}, 39(1):93--116, 1987.

\bibitem{Chang2011}
Chih~Chung Chang and Chih~Jen Lin.
\newblock {LIBSVM} : A library for support vector machines.
\newblock {\em ACM Transactions on Intelligent Systems and Technology},
  2(3):1--27, April 2011.

\bibitem{Chow1998}
Edmond Chow and Yousef Saad.
\newblock Approximate inverse preconditioners via sparse-sparse iterations.
\newblock {\em SIAM Journal of Scientific Computing}, 19(3):995--1023, 1998.

\bibitem{Davis:2011}
Timothy~A. Davis and Yifan Hu.
\newblock The {U}niversity of {F}lorida sparse matrix collection.
\newblock {\em {ACM} Transactions on Mathematical Software}, 38(1):1:1--1:25,
  2011.

\bibitem{Dembo1982}
Ron~S. Dembo, Stanley~C. Eisenstat, and Trond Steihaug.
\newblock Inexact {Newton} methods.
\newblock {\em SIAM Journal on Numerical Analysis}, 19(2):400--408, 1982.

\bibitem{Desoer1963}
C.~A. Desoer and B.~H. Whalen.
\newblock {A} note on pseudoinverses.
\newblock {\em Journal of the Society of Industrial and Applied Mathematics},
  11(2):442--447, 1963.

\bibitem{Feichtinger1991}
Hans~G. Feichtinger.
\newblock Pseudoinverse matrix methods for signal reconstruction from partial
  data.
\newblock {\em Proc. SPIE, Visual Communications and Image Processing '91},
  1606:766--772, 1991.

\bibitem{Gondzio2013}
Jacek Gondzio.
\newblock Convergence analysis of an inexact feasible interior point method for
  convex quadratic programming.
\newblock {\em SIAM Journal on Optimization}, 23(3):1510--1527, 2013.

\bibitem{GoulHribNoce01}
N.~I.~M. Gould, M.~E. Hribar, and J.~Nocedal.
\newblock On the solution of equality constrained quadratic problems arising in
  optimization.
\newblock {\em SIAM Journal on Scientific Computing}, 23(4):1375--1394, 2001.

\bibitem{Gould1998}
Nicholas I.~M. Gould and Jennifer~A. Scott.
\newblock Sparse approximate-inverse preconditioners using norm-minimization
  techniques.
\newblock {\em SIAM Journal on Scientific Computing}, 19(2):605--625, 1998.

\bibitem{GowerThesis}
Robert~M. Gower.
\newblock {\em Sketch and Project: Randomized Iterative Methods for Linear
  Systems and Inverting Matrices}.
\newblock PhD thesis, University of Edinburgh, 2016.

\bibitem{GowerGold2016}
Robert~M. Gower, Donald Goldfarb, and Peter Richt\'{a}rik.
\newblock Stochastic block {BFGS}: Squeezing more curvature out of data.
\newblock {\em Proceedings of the 33rd International Conference on Machine
  Learning}, 2016.

\bibitem{Gower2015c}
Robert~M. Gower and Peter Richt{\'{a}}rik.
\newblock Stochastic dual ascent for solving linear systems.
\newblock {\em arXiv:1512.06890}, 2015.

\bibitem{Gower2015}
Robert~Mansel Gower and Peter Richt{\'{a}}rik.
\newblock Randomized iterative methods for linear systems.
\newblock {\em SIAM Journal on Matrix Analysis and Applications},
  36(4):1660--1690, 2015.

\bibitem{Gower2016}
Robert~Mansel Gower and Peter Richt{\'{a}}rik.
\newblock Randomized quasi-{N}ewton updates are linearly convergent matrix
  inversion algorithms.
\newblock {\em arXiv:1602.01768}, 2016.

\bibitem{Grote96}
Marcus~J. Grote and Thomas Huckle.
\newblock Parallel preconditioning with sparse approximate inverses.
\newblock {\em SIAM J. Sci. Comput}, 18:838--853, 1996.

\bibitem{Loizourichtarik2016}
Nicolas Loizou and Peter Richt\'{a}rik.
\newblock A new perspective on randomized gossip algorithms.
\newblock In {\em 4th IEEE Global Conference on Signal and Information
  Processing (GlobalSIP)}, 2016.

\bibitem{Moore1920}
E.~H. Moore.
\newblock Abstract for ``{O}n the reciprocal of the general algebraic matrix''.
\newblock {\em Bulletin of the American Mathematical Society}, 26:394--395,
  1920.

\bibitem{PanSchreiber91}
Victor Pan and Robert Schreiber.
\newblock An improved newton iteration for the generalized inverse of a matrix,
  with applications.
\newblock {\em {SIAM} J. Scientific Computing}, 12(5):1109--1130, 1991.

\bibitem{Penrose1955}
R.~Penrose.
\newblock A generalized inverse for matrices.
\newblock {\em Mathematical Proceedings of the Cambridge Philosophical
  Society}, 51(3):406--413, 1955.

\bibitem{Robertson_apseudoinverse}
Gregory Robertson, R.~Lynn Kirlin, and W.~S. Lu.
\newblock A pseudoinverse update algorithm for rank-reduced covariance matrices
  from 2-d data.
\newblock {\em IEEE Signal Processing Letters}, 3, 1997.

\bibitem{Strohmer2009}
Thomas Strohmer and Roman Vershynin.
\newblock A randomized {K}aczmarz algorithm with exponential convergence.
\newblock {\em Journal of Fourier Analysis and Applications}, 15(2):262--278,
  2009.

\bibitem{Tapson2013}
J.~Tapson and A.~Van~Schaik.
\newblock Learning the pseudoinverse solution to network weights.
\newblock {\em Neural Networks}, 45:94--100, September 2013.

\end{thebibliography}
}

\section{Appendix}
Here we present and prove several fundamental linear algebra lemmas that are required to develop the main theorems in the paper.

\subsection{Key linear algebra lemmas}

\begin{lemma}\label{lem:09709s}For any matrix $W$ and symmetric positive semidefinite matrix $G$ such that
\begin{equation} \label{eq:Gnullassm}
 \Null{G} \subset \Null{W^\top },
 \end{equation} we have that
\begin{equation}\label{eq:8ys98hs}\Null{W} = \Null{W^\top G W}\end{equation}
and
\begin{equation}\label{eq:8ys98h986ss}\Range{W^\top } = \Range{W^\top G W}.\end{equation}
\end{lemma}
\begin{proof} 
In order to establish \eqref{eq:8ys98hs}, it suffices to show the inclusion $\Null{W} \supseteq \Null{W^\top G W}$ since the reverse inclusion trivially holds. Letting $s\in \Null{W^\top G W}$, we see that $\|G^{1/2}Ws\|^2=0$, which implies $G^{1/2}Ws=0$.
Consequently 
\[Ws \in \Null{G^{1/2}} = \Null{G} \overset{\eqref{eq:Gnullassm} }{\subset}  \Null{W^\top}.\] Thus $Ws \in \Null{W^\top} \cap \Range{W}$ which are orthogonal complements which shows that $Ws = 0.$

Finally, \eqref{eq:8ys98h986ss} follows from \eqref{eq:8ys98hs} by taking orthogonal complements. Indeed, $\Range{W^\top}$ is the orthogonal complement of $\Null{W}$ and $\Range{W^\top G W}$ is the orthogonal complement of $\Null{W^\top G W}$. \hfill\qed
\end{proof}

The following two lemmas are of key importance throughout the paper.
\begin{lemma} \label{lem:MMT}
For any matrix $M \in \R^{m \times n}$ and any matrix $R \in \R^{n \times d}$ such that $\Range{R} \subset \Range{M^\top}$ we have that
\begin{equation} \label{eq:MMT}
 \dotprod{ M^\top M R , R} \geq \lambda_{\min}^+(M^\top M) \dotprod{R,R},\end{equation}
\end{lemma}
\begin{proof}
Since 
\[\dotprod{ M^\top M R , R} = \Tr{R^\top M^\top M R} = \sum_{i=1}^d \dotprod{M^\top M R_{:i},R_{:i}},\]
the inequality~\eqref{eq:MMT} follows from the known inequality
\[ \dotprod{M^\top M v, v} \geq \lambda_{\min}^+(M^\top M) \dotprod{v,v},\]
where $v \in \Range{M^\top},$ which can be proved be diagonalizing $M^\top M.$
\end{proof}

\begin{lemma}\label{lem:WGWtight}
Let $0\neq W \in  \R^{m\times n}$ and
$G \in \R^{m\times m}$ be symmetric positive semi-definite with $\Null{G} \subset \Null{W^\top}$. Then the matrix $W^\top G W$ has a positive eigenvalue, and the following inequality holds:
\begin{equation}\label{eq:WGWtight}
 \dotprod{ W^\top G W R, R} \geq \lambda_{\min}^+(W^\top G W) \dotprod{R,R},
 \end{equation}
 where $R$ is a matrix with $n$ rows and $\Range{R} \subset \Range{W^\top}.$
\end{lemma}
\begin{proof}
By Lemma~\ref{lem:MMT} with $M = G^{1/2}W$ we have that~\eqref{eq:WGWtight} holds for $\Range{R} \subset \Range{W^\top G^{1/2}}.$ The proof now follows by observing 
\[ \Range{W^\top G^{1/2}} \overset{\text{Lemma}~\ref{lem:09709s} }{=} \Range{W^\top G W} \overset{\text{Lemma}~\ref{lem:09709s} }{=} \Range{W^\top}.\hfill\qed\]
\end{proof}


\subsection{Smallest nonzero eigenvalue of the product of two matrices}
\begin{lemma} \label{lem:prodsmallest}
Let $A, B \in \R^{n\times n}$ be symmetric positive semidefinite matrices. If 
\begin{equation}\label{eq:NullANullB}
\Null{A} \subset \Null{B}
\end{equation}
then
\begin{equation}\label{eq:prodsmallest}
\lambda_{\min}^+(AB) \geq \lambda_{\min}^+(A) \lambda_{\min}^+(B).
\end{equation} 
\end{lemma}
\begin{proof}
Using the variational formulation we have that
\begin{eqnarray}
\lambda_{\min}^+(AB) &= & \min_{v \in \Null{AB}^{\perp} } \frac{\norm{ABv}}{\norm{v}} \nonumber \\
&= & \min_{v \in \Null{AB}^{\perp} } \frac{\norm{ABv}}{\norm{Bv}} \frac{\norm{Bv}}{\norm{v}} \nonumber \\
& \leq & \min_{v \in \Null{AB}^{\perp} } \frac{\norm{ABv}}{\norm{Bv}}  \min_{v \in \Null{AB}^{\perp} }  \frac{\norm{Bv}}{\norm{v}}.\label{eq:lemlamAB}
\end{eqnarray}
Given that
\[ \Null{B} \subset \Null{AB}\]
ergo
\[\Null{AB}^\perp \subset \Null{B}^\perp \overset{\eqref{eq:NullANullB}}{\subset} \Null{A}^\perp. \]
The above shows that 
\begin{equation}
\min_{v \in \Null{AB}^{\perp} }  \frac{\norm{Bv}}{\norm{v}} \geq \min_{v \in \Null{B}^{\perp} }  \frac{\norm{Bv}}{\norm{v}} = \lambda_{\min}^+ (B). \label{eq:lemlamB}
\end{equation}
But also, since $\Null{B}^{\perp} = \Range{B}=\Range{BB}$ which follows from $B$ being symmetric and Lemma~\ref{lem:09709s}, we have that
\begin{eqnarray}
\min_{v \in \Null{AB}^{\perp} } \frac{\norm{ABv}}{\norm{Bv}}  
 & \geq  &  \min_{v \in \Range{B}} \frac{\norm{ABv}}{\norm{Bv}} \nonumber \\
&=& \min_{w \in \Range{BB}} \frac{\norm{Aw}}{\norm{w}} \nonumber \\
& = &  \min_{w \in \Range{B}} \frac{\norm{Aw}}{\norm{w}} = \lambda_{\min}^+ (A).
\label{eq:lemlamA}
\end{eqnarray}
Inserting~\eqref{eq:lemlamB} and~\eqref{eq:lemlamA} in~\eqref{eq:lemlamAB} gives the desired result.
\end{proof}

\subsection{Convenient probability lemma}

\begin{theorem} \label{lem:convprob}  Let $G$ be a positive symmetric semidefinite matrix. Let $S$ be a random matrix with a finite discrete distribution with $r$  $S= S_i \in \R^{n \times q_i}$ with probability  $p_i>0$ for $i=1,\ldots, r$. Let  $\mathbb{S} \eqdef [S_1, \ldots, S_r] \in \R^{n\times n}$. If
\begin{equation}\label{eq:convprob}p_i~=~\dfrac{\Tr{S_i^\top  G^2 S_i}}{\Tr{\mathbb{S}^\top  G^2 \mathbb{S}}},\quad \mbox{for } \quad i=1,\ldots, { r}.
\end{equation}
then
\begin{equation}
\lambda_{\min}^+ \left( G \E{S (S^\top G^2 S)^{\dagger}S^\top} G \right) \quad \geq \quad \frac{\lambda_{\min}^+\left( \mathbb{S}^\top  G^2\mathbb{S} \right) }{\Tr{\mathbb{S}^\top G^2\mathbb{S}}}.
\end{equation}
\end{theorem}
\begin{proof}
Let $Z  =G S (S^\top G^2 S)^{\dagger}S^\top G.$ Note that 
\begin{equation}
 \E{Z} = G\mathbb{S} D \mathbb{S}^\top G, \label{eq:EZlemappend}
\end{equation}
with 
\begin{equation}\label{eq:Ddeflemma}
D \quad \eqdef \quad \mbox{diag}(p_i (S^\top G^2 S)^{\dagger}).
\end{equation}
Let $t_i = \Tr{S_i^\top  G^2 S_i}$, and with~\eqref{eq:convprob} in~\eqref{eq:Ddeflemma} we have
\[ D =\frac{1}{\Tr{\mathbb{S}^\top G^2\mathbb{S}}}\mbox{diag}\left(t_1 (S_1^\top  G^2 S_1)^{\dagger}, \ldots,
t_r (S_r^\top  G^2 S_r)^{\dagger}\right),
 \]
thus
\begin{equation}\label{eq:Dlambda} \lambda_{\min}^+(D) =\frac{1}{\Tr{\mathbb{S}^\top G^2\mathbb{S}}} \min_{i \,: \, t_i \neq 0}\left\{ \frac{t_i}{\lambda_{\max}(S_i^\top  G^2 S_i)} \right\} \geq \frac{1}{\Tr{\mathbb{S}^\top G^2\mathbb{S}}}. 
\end{equation}
Thus
\begin{eqnarray}
\lambda_{\min}^+\left(\E{Z} \right)
&\overset{ \eqref{eq:EZlemappend}}{=} & \lambda_{\min}^+\left(G\mathbb{S} D \mathbb{S}^\top G\right) \nonumber\\
&= & \lambda_{\min}^+\left( D \mathbb{S}^\top G G\mathbb{S}\right) \nonumber\\
& \overset{\text{Lemma}~\ref{lem:prodsmallest} }{\geq} & \frac{\lambda_{\min}^+\left( \mathbb{S}^\top G G\mathbb{S}\right) }{\Tr{\mathbb{S}^\top G^2\mathbb{S}}},
\end{eqnarray} 
where in the second line we used that, for any matrices $A,B$, the matrices $AB$ and $BA$ share the same nonzero eigenvalues.\hfill\qed
\end{proof}

\end{document}